\documentclass[11pt]{article}
\def\thetitle{Discrepancies of Spanning Trees and Hamilton Cycles}

\usepackage{graphicx}

\usepackage{amsmath,amssymb}

\usepackage[usenames,dvipsnames,svgnames,table]{xcolor}
\definecolor{CombinatoricaAqua}{HTML}{00698C}
\definecolor{CombinatoricaBlue}{HTML}{3A3293}
\definecolor{CombinatoricaBrown}{HTML}{66220C}
\definecolor{CombinatoricaRed}{HTML}{DF2A27}
\definecolor{HarvardCrimson}{rgb}{0.6471, 0.1098, 0.1882}

\makeatletter
\let\reftagform@=\tagform@
\def\tagform@#1{\maketag@@@
	{(\ignorespaces\textcolor{CombinatoricaBrown}{#1}\unskip\@@italiccorr)}}
\renewcommand{\eqref}[1]{\textup{\reftagform@{\ref{#1}}}}
\makeatother

\usepackage[backref=page]{hyperref}
\hypersetup{%
	unicode,
	pdfencoding=auto,
	pdfauthor={Peleg Michaeli},
	pdftitle={\thetitle},
	pdfsubject={},
	pdfkeywords={},
	colorlinks=true,
	citecolor=CombinatoricaBlue,
	linkcolor=CombinatoricaAqua,
	anchorcolor=CombinatoricaBrown,
	urlcolor=HarvardCrimson}

\usepackage{amsthm}
\usepackage{thmtools}
\usepackage{bbm}
\usepackage{enumitem}
\usepackage[capitalize]{cleveref}
\Crefname{fact}{Fact}{Facts}
\Crefname{claim}{Claim}{Claims}


\declaretheoremstyle[
spaceabove=\topsep, spacebelow=\topsep,
headfont=\color{CombinatoricaBrown}\normalfont\bfseries,
bodyfont=\itshape,
]{thm}
\declaretheoremstyle[
spaceabove=\topsep, spacebelow=\topsep,
headfont=\color{CombinatoricaBrown}\normalfont\bfseries,
bodyfont=\normalfont,
]{dfn}
\declaretheoremstyle[
spaceabove=0.5\topsep, spacebelow=0.5\topsep,
headfont=\color{CombinatoricaBrown}\normalfont\bfseries,
bodyfont=\normalfont,
]{rmk}

\declaretheorem[style=thm,parent=section]{theorem}
\declaretheorem[style=thm,sibling=theorem]{lemma}
\declaretheorem[style=thm,sibling=theorem]{corollary}
\declaretheorem[style=thm,sibling=theorem]{claim}
\declaretheorem[style=thm,sibling=theorem]{proposition}

\declaretheorem[style=thm,sibling=theorem]{observation}

\declaretheorem[style=definition,numbered=no]{acknowledgements}

\usepackage[nobysame,msc-links]{amsrefs}

\renewcommand{\eprint}[1]{\href{https://arxiv.org/abs/#1}{arXiv:#1}}

\BibSpec{book}{%
	+{}  {\PrintPrimary}                {transition}
	+{,} { \textbf}                     {title} 
	+{.} { }                            {part}
	+{:} { \textit}                     {subtitle}
	+{,} { \PrintEdition}               {edition}
	+{}  { \PrintEditorsB}              {editor}
	+{,} { \PrintTranslatorsC}          {translator}
	+{,} { \PrintContributions}         {contribution}
	+{,} { }                            {series}
	+{,} { \voltext}                    {volume}
	+{,} { }                            {publisher}
	+{,} { }                            {organization}
	+{,} { }                            {address}
	+{,} { \PrintDateB}                 {date}
	+{,} { }                            {status}
	+{}  { \parenthesize}               {language}
	+{}  { \PrintTranslation}           {translation}
	+{;} { \PrintReprint}               {reprint}
	+{.} { }                            {note}
	+{.} {}                             {transition}
	+{}  {\SentenceSpace \PrintReviews} {review}
}
\BibSpec{incollection}{%
  +{}  {\PrintAuthors}                {author}
  +{,} { \textit}                     {title}
  +{.} { }                            {part}
  +{:} { \textit}                     {subtitle}
  +{,} { \PrintContributions}         {contribution}
  +{,} { \PrintConference}            {conference}
  +{}  {\PrintBook}                   {book}
  +{,} { }                            {booktitle}
	+{}  { \PrintEditorsB}              {editor}
	+{,} { }                            {publisher}
  +{,} { \PrintDateB}                 {date}
  +{,} { pp.~}                        {pages}
  +{,} { }                            {status}
  +{,} { \PrintDOI}                   {doi}
  +{,} { available at \eprint}        {eprint}
  +{}  { \parenthesize}               {language}
  +{}  { \PrintTranslation}           {translation}
  +{;} { \PrintReprint}               {reprint}
  +{.} { }                            {note}
  +{.} {}                             {transition}
  +{}  {\SentenceSpace \PrintReviews} {review}
}

\makeatletter
\renewcommand{\PrintNames@a}[4]{%
	\PrintSeries{\name}
	{#1}
	{}{ and \set@othername}
	{,}{ \set@othername}
	{}{ and \set@othername}
	{#2}{#4}{#3}%
}
\makeatother

\makeatletter
\def\mathcolor#1#{\@mathcolor{#1}}
\def\@mathcolor#1#2#3{%
	\protect\leavevmode
	\begingroup
	\color#1{#2}#3%
	\endgroup
}
\makeatother
\definecolor{Red}{rgb}{0.618,0,0}
\definecolor{Blue}{rgb}{0,0,1}
\definecolor{Green}{rgb}{0,0.298,0}

\usepackage{sectsty}
\sectionfont{\color{CombinatoricaBrown}}
\subsectionfont{\color{CombinatoricaBrown}}
\subsubsectionfont{\color{CombinatoricaBrown}}

\usepackage{soul}
\soulregister\ref7

\usepackage{pifont}
\usepackage{calc}

\usepackage[a4paper]{geometry}
\title{\thetitle}
\author{
  Lior Gishboliner\thanks{
    Department of Mathematics, ETH Z\"{u}rich, Switzerland.
    Email: \href{mailto:lior.gishboliner@math.ethz.ch}
                {\tt lior.gishboliner@math.ethz.ch}.}
  \and
  Michael Krivelevich\thanks{
    School of Mathematical Sciences,
    Tel Aviv University,
    Tel Aviv 6997801, Israel.
    Email: \href{mailto:krivelev@tauex.tau.ac.il}
                 {\tt krivelev@tauex.tau.ac.il}.
    Research supported in part by USA--Israel BSF grant 2018267
    and by ISF grant 1261/17.}
  \and
  Peleg Michaeli\thanks{
    School of Mathematical Sciences,
    Tel Aviv University,
    Tel Aviv 6997801, Israel.
    Email: \href{mailto:peleg.michaeli@math.tau.ac.il}
                 {\tt peleg.michaeli@math.tau.ac.il}.
    Research supported by ERC starting grant 676970 RANDGEOM
    and by ISF grant 1294/19.}
}

\makeatletter
\def\namedlabel#1#2{\begingroup
  #2%
  \def\@currentlabel{#2}%
  \phantomsection\label{#1}\endgroup
}
\makeatother

\usepackage{mleftright}
\mleftright

\newcommand{\defn}[1]{{\bfseries #1}}

\newcommand{\eps}{\varepsilon}
\renewcommand{\phi}{\varphi}

\newcommand{\cX}{\mathcal{X}}

\newcommand{\cC}{\mathcal{C}}
\newcommand{\cG}{\mathcal{G}}

\newcommand{\cT}{\mathcal{T}}

\newcommand{\floor}[1]{\left\lfloor{#1}\right\rfloor}
\newcommand{\ceil}[1]{\left\lceil{#1}\right\rceil}

\DeclareMathOperator{\bw}{bw}

\newcommand{\sm}{\setminus}
\newcommand{\es}{\varnothing}

\newcommand{\vect}{\mathbf}

\newcommand{\whp}[0]{\textbf{whp}}

\newcommand{\ind}{\mathbf{1}}

\usepackage{tabto}

\newcommand{\Disc}{\mathcal{D}}

\renewcommand{\C}{\mathcal{C}}

\newcommand{\T}{\mathcal{T}}
\newcommand{\Tn}{\mathcal{T}_n}

\usepackage{subcaption}
\usepackage{tikz}
\usetikzlibrary{calc}

\begin{document}
\maketitle

\begin{abstract}
We study the multicolour discrepancy of spanning trees and Hamilton cycles in graphs.
As our main result, we show that under very mild conditions, the $r$-colour spanning-tree discrepancy of a graph $G$ is equal, up to a constant, to the minimum $s$ such that $G$ can be separated into $r$ equal parts by deleting $s$ vertices.
This result arguably resolves the question of estimating the spanning-tree discrepancy in essentially all graphs of interest.
In particular, it allows us to immediately deduce as corollaries most of the results that appear in a recent paper of Balogh, Csaba, Jing and Pluh\'{a}r, proving them in wider generality and for any number of colours. We also obtain several new results, such as determining the spanning-tree discrepancy of the hypercube.
For the special case of graphs possessing certain expansion properties, we obtain exact asymptotic bounds.

We also study the multicolour discrepancy of Hamilton cycles in graphs of large minimum degree, showing that in any $r$-colouring of the edges of a graph with $n$ vertices and minimum degree at least $\frac{r+1}{2r}n + d$, there must exist a Hamilton cycle with at least $\frac{n}{r} + 2d$ edges in some colour. This extends a result of Balogh et al., who established the case $r = 2$. The constant $\frac{r+1}{2r}$ in this result is optimal; it cannot be replaced by any smaller constant.
\end{abstract}

\section{Introduction}
\label{sec:intro}

Combinatorial discrepancy theory aims to quantify the following phenomenon: if a hypergraph $\mathcal{H}$ is ``sufficiently rich'', then in every $2$-colouring of the vertices of $\mathcal{H}$ there will be some hyperedge which is unbalanced, namely, has significantly more vertices in one of the colours than in the other.
The corresponding parameter, called the discrepancy of $\mathcal{H}$, is then defined as the maximum unbalance that is guaranteed to occur (on some hyperedge) in every $2$-colouring of $V(\mathcal{H})$. More concretely, assuming that the colours are $-1$ and $1$, we can define the unbalance of a hyperedge $e$ under a colouring $f : V(\mathcal{H}) \rightarrow \{-1,1\}$ to be $f(e) := \left| \sum_{x \in e}{f(x)} \right|$, and the discrepancy is then the minimum of $\max_e f(e)$ over all colourings $f$.
The study of such problems has a long and rich history, with several influential results.
We refer the reader to Chapter~4 in the book of Matou\v{s}ek~\cite{Mat} for a thorough overview.

In this paper, we are concerned with discrepancy questions in the context of graphs. In this setting, the vertices of the hypergraph $\mathcal{H}$ are the edges of some graph $G$, and the hyperedges of $\mathcal{H}$ correspond to subgraphs of $G$ of a particular type, such as spanning trees, cliques, Hamilton cycles, clique factors, etc. There are several classical results in this vein for the case that $G$ is a complete graph, including those of Erd\H{o}s and Spencer~\cite{ES72} and Erd\H{o}s, F\"{u}redi, Loebl and S\'{o}s~\cite{EFLS95}.
Recently, Balogh, Csaba, Jing and Pluh\'{a}r~\cite{BCJP20} (see also~\cite{BCPT21}) initiated the study of discrepancy problems for arbitrary graphs $G$, focusing on the discrepancy of spanning trees and Hamilton cycles.
In the present paper, we continue this study.
Our main result is a very general theorem on the discrepancy of spanning trees, which arguably resolves the problem of its estimation for all $3$-vertex-connected graphs (as well as all ``sufficiently expanding'' $2$-vertex-connected graphs; see the next section for the precise definition).

Our results apply to the more general setting of {\em multicolour discrepancy}.
While there are several natural ways to generalise the above definition of $2$-colour discrepancy to an arbitrary number of colours, the resulting parameters are all within a multiplicative factor of each other, making the choice mostly a matter of convenience.
Here we have chosen to use the following definition. For a hypergraph $\mathcal{H}$, an $r$-colouring
$f : V(\mathcal{H}) \rightarrow [r]$ of its vertices and a hyperedge $e \in E(\mathcal{H})$, define the unbalance of $e$ with respect to $f$ to be
\[
D_f(e) := r \cdot \left( \max_{1 \leq i \leq r}|f^{-1}(i) \cap e| - \frac{|e|}{r} \right).
\]
In other words, $D_f(e)$ measures (up to a scaling factor of $r$) the difference between the largest number of vertices of $e$ in a given colour and the average (which is $|e|/r$). Multiplying this difference by $r$ (as is done above) produces a more convenient definition (which always gives integer values).
For a colouring $f : V(\mathcal{H}) \rightarrow [r]$, define
\[
\Disc(\mathcal{H},f) = \max_{e \in E(\mathcal{H})}{D_f(e)}.
\]
The \defn{$r$-colour discrepancy} of $\mathcal{H}$ is then defined as
\[
\mathcal{D}_r(\mathcal{H}) :=
\min_{f : V(\mathcal{H}) \rightarrow [r]}
\Disc_r(\mathcal{H},f).
\]
Note that $\mathcal{D}_2(\mathcal{H})$ coincides with the ($2$-colour) discrepancy of $\mathcal{H}$ defined in the beginning of this section.
For a graph $G$ and a set $\cX$ of subgraphs of $G$, we define $\Disc_r(G,\cX)$ to be the $r$-colour discrepancy of the hypergraph $\mathcal{H}$ with vertex-set $V(\mathcal{H}) = E(G)$ and edge-set $E(\mathcal{H}) = \cX$. We will also sometimes use the notation $\Disc(G,\cX,f)$, which is analogous to $\Disc(\mathcal{H},f)$.
It is worth mentioning that discrepancy-type questions for more than two colours already appear in the literature, see e.g.~\cite{DS03}. Moreover, very recently, such questions have also been considered in the context of discrepancy of graphs. Specifically, the multicolour discrepancy of Hamilton cycles in random graphs has been studied in~\cite{GKM20+}.

\subsection{Discrepancy of Spanning Trees}
Spanning trees are among the most basic objects in graph theory.
Let us denote the set of all trees on $n$ vertices by $\T_n$.
Hence, for an $n$-vertex graph $G$, $\Disc_r(G,\T_n)$ denotes the $r$-colour discrepancy of spanning trees of $G$.

We now introduce a graph parameter which will play a central role in our results in this section. For an integer $r \geq 2$ and a graph $G$, denote by $s_r(G)$ the minimum $s$ for which there is a partition $V(G) = V_1 \cup \dots \cup V_r \cup S$ such that $|V_1| = \dots = |V_r|$, $|S| = s$, and there are no edges between $V_i$ and $V_j$ for any $1 \leq i < j \leq r$.
Such a partition is called a \defn{balanced $r$-separation} of $G$,
a set $S$ as above is called a \defn{balanced $r$-separator of $G$}, and $s_r(G)$ will be referred to as the \defn{balanced $r$-separation number} of $G$.

Balogh, Csaba, Jing and Pluh\'{a}r~\cite{BCJP20} observed that the $2$-colour spanning-tree discrepancy of $G$ is no larger than $s_2(G) - 1$.
This fact easily generalises to $r$ colours.
Indeed, given an $n$-vertex graph $G$ and a partition
$V(G) = V_1 \cup \dots \cup V_r \cup S$ as above,
consider the $r$-colouring $f:E(G)\to[r]$ defined
by assigning colour~$i$ to all edges which intersect $V_i$ ($i = 1,\dots,r$), and colouring the edges contained in $S$ arbitrarily.
Observe that if $T$ is a spanning tree of $G$,
then for every $i\in[r]$, the forest $T[V_i\cup S]$ has at least $|V_i|$ edges touching $V_i$, hence at least $|V_i|$ edges coloured $i$.
Since the total number of edges of $T$ is $n-1$, we have that the size of a maximum colour class is at most $(n-1)-(r-1)|V_1|=|S|-1+(n-|S|)/r$, hence
$\Disc_r(G,\T_n)\le r(|S|-1+(n-|S|)/r-(n-1)/r) = (r-1)(|S|-1)$.
This shows that
\begin{equation}\label{eq:ub}
  \Disc_r(G,\T_n) \leq (r-1)(s_r(G) - 1).
\end{equation}

Given \eqref{eq:ub}, it is natural to ask to which extent $s_r(G)$ ``controls" $\Disc_r(G,\mathcal{T}_n)$.
Unfortunately, these two parameters might be arbitrarily far apart.
In fact, it is not hard to construct graphs on $n$ vertices with $s_r(G)=\Theta(n)$ but $\Disc_r(G,\Tn)\le 1$.
Indeed, consider the following family of graphs.
A \defn{hedgehog} with proportion $r$ on $n$ vertices is a clique (``body'') on $n/r$ vertices (assuming $r$ divides $n$), each is connected to $(r-1)$ distinct vertices outside the clique (``spikes''; see \cref{fig:hedgehog,fig:hedgehog3}).
It is not hard to see that any balanced $r$-separator of the hedgehog is of linear size.
By colouring its body with the colour~$r$ and the $r-1$ spikes emerging from each vertex of the body by colours $1,\ldots,r-1$, one may verify that the $r$-colour spanning-tree discrepancy of the hedgehog is $1$.
This construction can be generalized to obtain graphs of large minimum degree (and even degree-regular ones) (see~\cref{fig:hedgehog5}) which still have the property that their $r$-separation number is $\Theta(n)$ while their spanning-tree discrepancy is only $O(1)$.%
\footnote{Indeed, for $d\ge 5$ one may construct such a $d$-regular ``hedgehog'' on $n$ vertices as follows.
Let $m=n/(d+2)$ and let $H_1,\ldots,H_m$ be disjoint $(d+1)$-cliques.
Remove a single edge from each clique, $H_i$ and connect both of its endpoints to a new vertex $v_i$.
Finally, endow the set $v_1,\ldots,v_m$ with a $(d-2)$-expander.}
However, a common notable property of all of these examples is that their vertex connectivity is $1$, namely, that they are not $2$-connected.%
\footnote{Here and later, by $k$-connected, we mean $k$-vertex-connected.}

\begin{figure*}[t!]
  \captionsetup{width=0.879\textwidth,font=small}
  \centering
  \begin{subfigure}[t]{0.29\textwidth}
    \centering
    \includegraphics{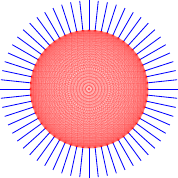}
    \caption{A hedgehog on $120$ vertices, whose body is coloured red and whose spikes are coloured blue.}
    \label{fig:hedgehog}
  \end{subfigure}%
  ~
  \begin{subfigure}[t]{0.04\textwidth}
    ~
  \end{subfigure}
  \begin{subfigure}[t]{0.29\textwidth}
    \centering
    \includegraphics{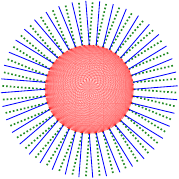}
    \caption{A hedgehog on $120$ vertices, whose body is coloured red and whose spikes are coloured green and blue.}
    \label{fig:hedgehog3}
  \end{subfigure}
  \begin{subfigure}[t]{0.04\textwidth}
  ~
  \end{subfigure}
  \begin{subfigure}[t]{0.29\textwidth}
    \centering
    \includegraphics{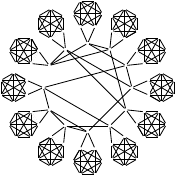}
    \caption{A $5$-regular ``hedgehog'' on $84$ vertices.
      Its body is a random $3$-regular graph on $12$ vertices.}
    \label{fig:hedgehog5}
  \end{subfigure}
  \caption{Hedgehogs.}
  \label{fig:hedgehogs}
\end{figure*}

Our main result shows that already the (rather weak) requirement of $3$-connectivity guarantees that there is a strong relation between balanced $r$-separations and $r$-colour spanning-tree discrepancy, namely, that these two parameters are a constant factor apart from each other. The same conclusion is valid for $2$-connected graphs in which $s_r(G)$ is large enough.
These statements are given in the following theorem.

\begin{theorem}\label{thm:trees}
	For every $r \geq 2$ there exists $C = C(r) > 0$ such that the following holds. Let $G$ be an $n$-vertex graph satisfying one of the following conditions:
	\begin{itemize}
	    \item $G$ is $3$-connected,
	    \item $G$ is $2$-connected and $s_r(G) \geq C\sqrt{n}$.
	\end{itemize}
  Then $\Disc_r(G,\mathcal{T}_n) \geq \floor{s_r(G)/C}$.
\end{theorem}

\Cref{thm:trees} can be interpreted as saying that under very mild assumptions,
having small balanced separations is the {\em only} cause of small spanning-tree discrepancy.

We also show that the lower-bound condition on $s_r(G)$ in the second item of \cref{thm:trees} is essentially tight, as there are $2$-connected $n$-vertex graphs $G$ with $s_r(G) = \Omega(\sqrt{n})$ and $\Disc_r(G,\mathcal{T}_n) \leq 1$.

\begin{proposition}\label{prop:tightness}
  For every $r\ge 2$ there exists $c=c(r)>0$ such that for infinitely many integers $n$, there exists an $n$-vertex $2$-connected graph $G$ with $s_r(G)\ge c\sqrt{n}$ and $\Disc_r(G,\Tn)\le 1$.
\end{proposition}

\Cref{thm:trees} allows us to determine the spanning-tree discrepancy for many graphs of interest.
In particular, it immediately implies all results of~\cite{BCJP20} concerning the discrepancy of spanning trees (up to constant factors) and generalises them to any number of colours.
Several new results can also be obtained.
Below we give a representative sample of such corollaries.

When applying \cref{thm:trees}, we need to be able to lower-bound the balanced $r$-separation number $s_r$ of the graphs in question. As it turns out, it will often be more convenient to lower-bound other graph parameters, which are in turn lower bounds for $s_r$. One such parameter is the following.
For a graph $G$, let $\iota(G)$ denote its \defn{vertex isoperimetric constant}, namely, the minimum of $|N(U)|/|U|$, taken over all sets $U\subseteq V(G)$ with $0<|U|\le |V(G)|/2$, where $N(U)$ denotes the \defn{external neighbourhood} of $U$ (namely, the set of vertices outside $U$ which have a neighbour in~$U$).
It is not hard to see that the balanced $r$-separation number of any graph is at least of the same order of magnitude as $n$ times its vertex isoperimetric constant.
Indeed, given a balanced $r$-separation $V_1\cup\ldots\cup V_r\cup S$ of $G$, the size of $V_1$ is $(n-|S|)/r$ and its neighbourhood is contained entirely in $S$, meaning that $r|S|/(n-|S|)\ge\iota(G)$ and hence $|S|=\Omega(\iota(G)\cdot n)$. Thus, \cref{thm:trees} has the following useful corollary:
\begin{corollary}[Isoperimetry]\label{cor:iso}
For every $r\ge 2$ there is $C = C(r)>0$ such that the following holds. Let $G$ be an $n$-vertex graph and suppose either that $G$ is $3$-connected, or that $G$ is $2$-connected and $\iota(G)\ge C n^{-1/2}$. Then $\Disc_r(G,\Tn)\ge \floor{\frac{\iota(G)\cdot n}{C}}$.
\end{corollary}

Before proceeding to applications of \cref{cor:iso}, let us quickly mention another
corollary of a similar flavour, stating that highly-connected graphs have large spanning-tree discrepancy. Denote by $\kappa(G)$ the vertex connectivity of $G$.

\begin{corollary}[Discrepancy vs.\ vertex-connectivity]\label{cor:conn}
  For every $r\ge 2$ there is ${C= C(r)> 0}$ such that for every connected $n$-vertex graph $G$ it holds that
  $\Disc_r(G,\Tn)\ge \floor{\kappa(G)/C}$.
\end{corollary}

\Cref{cor:conn} follows immediately from \cref{thm:trees}, since $s_r(G) \geq \kappa(G)$ for every graph~$G$.
Similarly, \cref{thm:trees} implies that if $G$ is $3$-connected and $\delta(G) \geq n/r + k$ for some $k > 0$, then $\Disc_r(G,\Tn)\ge \lfloor k/C \rfloor$, since for every graph $G$ it holds that $\delta(G) \leq \frac{n-s_r(G)}{r} - 1 + s_r(G) \le \frac{n}{r} + s_r(G)$. In the case $r = 2$, one does not need the $3$-connectivity assumption, because every graph $G$ with $\delta(G)\ge n/2+k$ is $(2k+1)$-connected.

To see that \cref{cor:conn} is tight, consider the graph on the vertex set $V$ with the balanced $r$-separation $V=V_1\cup\ldots\cup V_r\cup S$ having $|S|=k\le n/(r+1)$, and endow the graph with all possible edges except for those connecting $V_i$ to $V_j$ for $i\ne j$.
This graph is clearly $k$-connected, but according to \eqref{eq:ub}, its spanning-tree discrepancy is at most $O(k)$.

To demonstrate the usefulness of \cref{cor:iso}, let us apply it to estimate the spanning-tree discrepancy of random regular graphs.
For an integer $d\ge 3$, let $\cG_{n,d}$ denote the uniform distribution over the set of all $d$-regular graphs on $n$ vertices (assuming $dn$ is even).
Balogh et al.~\cite{BCJP20}*{Theorem~3} have shown that  $\Disc_2(\cG_{n,d},\Tn)=\Theta(n)$ \whp.
Here we immediately obtain an extension of this result to any number of colours.
\begin{corollary}[Random regular graphs]\label{cor:rrg}
	Let $G\in\cG_{n,d}$, $d\ge 3$. Then $\Disc_r(G,\Tn)=\Theta(n)$ \whp{}.
\end{corollary}
\Cref{cor:rrg} follows from \cref{cor:iso} by recalling that \whp{}, $G\sim \cG_{n,d}$ is $3$-connected (see, e.g.,~\cite{FK}) and satisfies $\iota(G)\ge\eta$ for some suitable $\eta>0$ (see~\cite{Bol88}).

For large $r$ we can go a step further, determining the asymptotics, as a function of $r$, of the multiplicative constant appearing in \cref{cor:rrg}. This is stated in the following theorem:
\begin{theorem}\label{prop:rrd_tight}
Let $G\in\cG_{n,d}$, $d\ge 3$. Then $\Disc_r(G,\Tn)=
\left( \frac{d}{2} - 1 - o_r(1) \right)n$ \whp{}. In other words, \whp{}
in every $r$-colouring of $E(G)$ there is a spanning tree with at least $\left( \frac{d}{2r} - o\left( \frac{1}{r} \right) \right) n$ edges of the same colour, and the constant $\frac{d}{2r}$ is tight.
\end{theorem}

Results similar to \cref{cor:rrg} can be obtained for regular {\em expander graphs}.
Let $G$ be a $d$-regular $n$-vertex graph, and let $\lambda=\lambda(G)$ be the second largest eigenvalue of its adjacency matrix.
It is widely known that a small ratio $\lambda/d$ implies good expansion properties (for a survey we refer the reader to~\cite{HLW06}).
In particular, it follows from~\cite{AM85}*{Lemma 2.1} (see also~\cite{AS}*{Theorem~9.2.1}) that if $d$ is constant and $\lambda\le d-2$ then $G$ is $2$-connected and $\iota(G)\ge 1/d$.
This, together with \cref{cor:iso}, implies the following.
\begin{corollary}[Regular expanders]\label{cor:expanders}
  Let $G$ be a $d$-regular $n$-vertex graph with $d\ge 3$ and $\lambda\le d-2$.
  Then $\Disc_r(G,\Tn)= \Theta_d(n)$.
\end{corollary}
We remark that a similar statement holds for $d\gg 1$ as well by strengthening the assumption to $\lambda\le(1-\eps)d$ for some $\eps>0$.

As our next application, we determine the spanning-tree discrepancy of the $d$-dimensional hypercube, denoted here by $Q_d$.
\begin{corollary}[The hypercube]\label{cor:hypercube}
  $\Disc_r(Q_d,\Tn) = \Theta( n/\sqrt{\log{n}} )$ where $n=2^d$.
\end{corollary}
The derivation of the lower bound in \cref{cor:hypercube} from \cref{thm:trees} follows the same lines as the derivation of the preceding corollaries.
Naturally, we require estimates for the vertex isoperimetric constant of the hypercube. Such estimates are indeed available~\cites{Har66}.
The details are given in \cref{sec:strees:app}.

For our final application, we let $P_k^d$ denote the $d$-dimensional grid on $k^d$ vertices (${d \ge 2}$).
Balogh et al.~\cite{BCJP20}*{Theorem~1.5} have shown that $\Disc_2(P_k^2,\Tn) =\Theta( \sqrt{n})$.
Here we obtain a generalisation of this result to every $d \geq 2$ and every number of colours.
\begin{corollary}[$d$-dimensional grids]\label{cor:grid}
  $\Disc_r(P_k^d,\Tn)=\Theta( k^{d-1} )$ where $n=k^d$, $d\ge 2$.
\end{corollary}
Again, the proof is achieved by combining \cref{cor:iso} with a suitable isoperimetric inequality. Such an inequality for the grid was given in~\cite{BL91v}. The details appear in \cref{sec:strees:app}.
In fact, our methods yield similar results for a much wider family of ``grid-like" graphs, such as tori (of dimension $d\ge 2$), hexagonal and triangular lattices, etc.

~

It is natural to ask about the spanning-tree discrepancy of the complete graph $K_n$. Since discrepancy is monotone with respect to adding edges, this is also the maximum possible spanning-tree discrepancy that an $n$-vertex graph can have.
As it turns out, the $r$-colour spanning-tree discrepancy of $K_n$ is closely related to a certain parameter $\varphi(r)$, defined in terms of covering the edges of a complete graph by smaller complete graphs.
The definition of $\varphi(r)$ is as follows.

Let $\phi(r,n)$ denote the smallest integer $k$ such that there is a covering of the edges of $K_n$ with $r$ cliques of size $k$.
In other words, $\phi(r,n)$ is the smallest integer $k$ for which there is a collection of $k$-sets $A_1,\dots,A_r \in \binom{[n]}{k}$ such that every $e \in \binom{[n]}{2}$ is contained in $A_i$ for some $i\in[r]$.%
\footnote{We remark that this definition can be extended to arbitrary graphs, and that the resulting parameter is related to $s_r(G)$.
Indeed, let $\phi(r,G)$ be the smallest integer $k$ such that there exist $A_1,\ldots,A_r\in\binom{V(G)}{k}$ for which for every $e\in E(G)$ there is $j\in[r]$ with $e\subseteq A_j$.
It is not hard to verify that $s_r(G) = \Theta_r(\phi(r,G)-n/r)$.}
This parameter has been studied by Mills~\cite{Mil79} and by Hor\'{a}k and Sauer~\cite{HS92}. In these works it was shown that the limit $\phi(r):=\lim_{n\to\infty} \phi(r,n)/n$ exists and is equal to
$\min_{n} \phi(r,n)/n$ (this minimum is attained for infinitely many $n$). In particular, $\varphi(r,n) \geq \varphi(r) \cdot n$ for every $n$.
The value of $\varphi(r)$ for small $r$ was determined in \cites{HS92,Mil79};
for example, $\varphi(2) = 1$, $\varphi(3) = \frac{2}{3}$, $\varphi(4) = \frac{3}{5}$, $\varphi(5) = \frac{5}{9}$, $\varphi(6) = \frac{1}{2}$ and $\varphi(7) = \frac{3}{7}$
(for the values of $\varphi(r)$ for $r \leq 13$, see~\cite{Mil79}).
A trivial counting argument shows that $\varphi(r) \geq \frac{1}{\sqrt{r}}$;
indeed, if $A_1,\dots,A_r \in \binom{[n]}{k}$ cover all pairs in $\binom{[n]}{2}$, then $r \binom{k}{2} \geq \binom{n}{2}$, which gives $k \geq \frac{n}{\sqrt{r}} - o(n)$.
On the other hand, if $r = p^2 + p + 1$ and a projective plane of order $p$ exists, then the value of $\varphi(r)$ is known exactly: $\varphi(r) = \frac{p+1}{p^2+p+1}$
(see~\cite{Fur88}*{Section~7} and the references therein).
The construction of $A_1,\dots,A_r$ as above is obtained by blowing up a projective plane.
By using this result together with known facts about the existence of projective planes,
one can show that
$\varphi(r) = \frac{1 + o_r(1)}{\sqrt{r}}$ for every $r$.

Observe that for every $n$-vertex graph $G$, one can $r$-colour the edges of $G$ so that no spanning tree of $G$ has more than $\varphi(r,n) - 1$ edges of the same colour. Indeed, setting $k = \varphi(r,n)$, take $A_1,\dots,A_r \in \binom{[n]}{k}$ as in the definition of $\varphi(r,n)$, and colour an edge $e \in E(G)$ with colour $i \in [r]$ if $e \subseteq A_i$; such an $i$ always exists because $A_1,\dots,A_r$ cover all pairs in $\binom{[n]}{2}$.
Observe that in this colouring, every edge of colour~$i$ is contained in $A_i$, meaning that any spanning tree of $G$ contains at most $|A_i| - 1 = \varphi(r,n) - 1$ edges of colour~$i$, as required.

In the other direction, one can show that in any $r$-colouring of the edges of $K_n$, there is a spanning tree with at least
$(\varphi(r) - o(1))n$ edges of the same colour (we shall prove this as part of a more general result). The construction in the previous paragraph
shows that the constant $\varphi(r)$ is optimal. It would be interesting to obtain an exact result.
One might wonder whether the upper bound $\varphi(r,n) - 1$ is tight, namely, whether it is true that in every $r$-colouring of $E(K_n)$ there is a spanning tree with $\varphi(r,n) - 1$ edges of the same colour.

As our next theorem, we show that the spanning-tree discrepancy of graphs with certain expansion properties is essentially as high as that of $K_n$. In other words, we show that the optimal bound $(\varphi(r) - o(1))n$ holds for these graphs as well.
The precise notion of expansion is as follows:
say that a graph $G = (V,E)$ is a \defn{$\beta$-graph} (for a given
$\beta>0$) if $|V| \geq 1/\beta$ and there is an edge in $G$ between every pair of disjoint sets
$U,W\subseteq V$ with $|U|,|W|\ge\beta |V|$ (cf.~\cite{FK21}).
One remarkable class of such expanders is $(n,d,\lambda)$-graphs, where $\lambda/d\le\beta^2$ (for an overview of $(n,d,\lambda)$-graphs, we refer the reader to \cite{KS06}).
Thus, the following theorem gives an example of sparse (i.e., having linearly many edges) graphs with nearly-optimal spanning-tree discrepancy.
We note that a $\beta$-graph $G$ need not be connected, as for example it may have up to $\beta |V(G)| - 1$ isolated vertices. Hence, when studying the spanning-tree discrepancy of $\beta$-graphs, we need to explicitly assume that the graphs in question are connected.

\begin{theorem}[$\beta$-graphs]
  \label{thm:beta}
  For every $r \geq 2$ and $\eps > 0$ there is $\beta = \beta(r,\eps) > 0$ such that every connected $n$-vertex $\beta$-graph $G$ satisfies the following: in any $r$-colouring of $E(G)$ there is a spanning tree with at least $(\varphi(r) - \eps) \cdot n$ edges of the same colour.
\end{theorem}

In light of the above discussion, \cref{thm:beta} can be interpreted as saying that $\beta$-graphs essentially achieve the maximum possible $r$-colour spanning-tree discrepancy of any graph on the same number of vertices.

It is worth noting that a relation of a similar flavour --- that is, between a ``covering-pairs" type parameter and a multicolour Ramsey-type problem --- was demonstrated in~\cite{DGS21}.

\subsection{Discrepancy of Hamilton Cycles}
Hamilton cycles are among the most well-studied objects in graph theory, boasting many hundreds of papers. Here we study the multicolour discrepancy of Hamilton cycles in dense graphs.
One of the main results of~\cite{BCJP20} establishes that for every $\eps > 0$, every $n$-vertex graph $G$ with minimum degree at least $(\frac{3}{4} + \eps)n$ satisfies
$\Disc_2(G,\mathcal{H}_n) = \Omega(n)$, and that moreover, the fraction $\frac{3}{4}$ is best possible. Here we generalise this result to any number of colours.
\begin{theorem}\label{thm:ham}
	Let $r \geq 2$ and $0 \leq d \leq \frac{n}{28r^2}$, and let $G$ be a graph with $n \geq n_0(r)$ vertices and minimum degree at least
	$\frac{(r+1)n}{2r} + d$. Then in every $r$-colouring of the edges of $G$ there is a Hamilton cycle with at least
	$\frac{n}{r} + 2d$
	edges of the same colour.
\end{theorem}
We remark that the same result has been very recently independently obtained by Freschi, Hyde, Lada and Treglown~\cite{FHLT21}. Our proof is entirely different from the one given in~\cite{FHLT21}, and gives a slightly better dependence of the bound on $d$.

The constant $\frac{r+1}{2r}$ in \cref{thm:ham} is optimal, that is, it cannot be replaced with any smaller constant. Indeed, we now describe an $n$-vertex graph with minimum degree $\frac{(r+1)n}{2r}$ (assuming $2r$ divides $n$) which may be $r$-coloured such that each Hamilton cycle has exactly $\frac{n}{r}$ edges of each colour.\footnote{The same construction appears in~\cite{FHLT21}.}
Let $G$ be the graph on the vertex set $V=V_1\cup\ldots\cup V_r$,
where $V_1,\ldots,V_r$ are disjoint sets, $|V_i|=\frac{n}{2r}$ for $i<r$ and $|V_r|=\frac{(r+1)n}{2r}$.
The edges of $G$ are all pairs touching $V_r$.
Note that the minimum degree of $G$ is $\delta(G)=|V_r|=\frac{(r+1)n}{2r}$.
For each $1 \leq i \leq r-1$, we colour the edges between $V_i$ and $V_r$ in colour~$i$, and the rest of the edges in colour~$r$ (See~\cref{fig:ham}).
It is easy to see that any Hamilton cycle in $G$ has two edges touching every vertex in any $V_i$, $i=1,\ldots,r-1$, and these edges are distinct for distinct vertices.
Thus, the number of edges in any colour is exactly $\frac{n}{r}$.

In the same construction, every perfect matching has exactly $\frac{n}{2r}$ edges touching $V_i$ ($i=1,\ldots,r-1$), and thus coloured $i$; this leaves exactly $\frac{n}{2r}$ edges for colour~$r$.
On the other hand, looking again at \cref{thm:ham}, under its assumptions, we are guaranteed to find a Hamilton cycle with at least one biased colour (colour~$1$, say), namely, in which at least $\frac{n}{r} + 2d$ edges are coloured $1$.
If $n$ is even, this Hamilton cycle can be decomposed into two perfect matchings; one of these perfect matchings will have at least $\frac{n}{2r} + d$ edges in colour~$1$.
Thus, we obtain the following result:
\begin{corollary}\label{cor:PM}
  	Let $r \geq 2$ and $0 \leq d \leq \frac{n}{28r^2}$, and let $G$ be a graph with $n \geq n_0(r)$ vertices ($n$ even) and minimum degree at least
	$\frac{(r+1)n}{2r} + d$. Then in every $r$-colouring of the edges of $G$ there is a perfect matching with at least
	$\frac{n}{2r} + d$
	edges of the same colour.
\end{corollary}
\noindent
The constant $\frac{r+1}{2r}$ in \cref{cor:PM} is again optimal, as shown by the above example.
Finally, let us note that in Theorem \ref{thm:ham} and Corollary \ref{cor:PM}, $n_0(r)$ depends (only) polynomially on $r$.

\begin{figure*}[t!]
  \captionsetup{width=0.879\textwidth,font=small}
  \centering
  \includegraphics{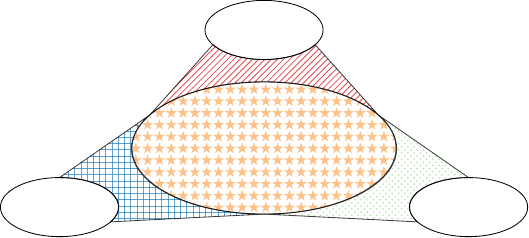}
  \caption{A $4$-coloured graph on $n$ vertices.
    Each of the small bulbs represents a set of $n/8$ vertices, and the large bulb in the middle represents a set of $5n/8$ vertices.
    In this graph every Hamilton cycle has exactly $n/4$ edges coloured in each of the four colours, and, assuming $n$ is even, every perfect matching has exactly $n/8$ edges coloured in each of the four colours.}
  \label{fig:ham}
\end{figure*}

\paragraph{Organisation}
In \cref{sec:strees_main} we prove \cref{thm:trees,prop:tightness}.
\Cref{prop:rrd_tight} is proved in \cref{sec:rrg},
and in \cref{sec:strees:app} we give the full details of the proofs of \cref{cor:hypercube,cor:grid}.
\Cref{thm:beta} is proved in \cref{sec:beta},
and finally in \cref{sec:HamCycles} we establish \cref{thm:ham}.

\paragraph{Notation and terminology}
Let $G=(V,E)$ be a graph.
For two vertex sets $U,W\subseteq V$ we denote by $E_G(U)$ the set of edges of $G$ spanned by $U$ and by $E_G(U,W)$ the set of edges having one endpoint in $U$ and the other in $W$.
The degree of a vertex $v\in V$ is denoted by $d_G(v)$, and we write $d_G(v,U)=|E_G(\{v\},U)|$.
We let $\delta(G)$ and $\Delta(G)$ denote the minimum and maximum degrees of $G$.
When the graph $G$ is clear from the context, we may omit the subscript $G$ in the notations above.

If $f,g$ are functions of $n$ we write $f\preceq g$ if $f=O(g)$ and $f\succeq g$ if $f=\Omega(g)$.
For simplicity and clarity of presentation, we often make no particular effort to optimise the constants obtained in our proofs and omit floor and ceiling signs whenever they are not crucial.

\section{Proof of \cref{thm:trees} and \cref{prop:tightness}}\label{sec:strees_main}
The goal of this section is to prove \cref{thm:trees,prop:tightness}, starting with the former.
Let us introduce some definitions and terminology which will be used in the proof.
    Let $r \geq 2$, let $G$ be a graph, and let
    $f : E(G) \rightarrow [r]$ be an $r$-colouring of the edges of $G$. For each $1 \leq i \leq r$, let $G_i$ be the spanning graph of $G$ consisting of the edges of colour~$i$. Connected components of $G_i$ will be called \defn{colour-$i$ components}. We use $\C_i$ to denote the set of all colour-$i$ components. For a vertex $v \in V(G)$, we denote by $C_i(v)$ the unique colour-$i$ component containing $v$.
	Crucially, define $H = H(G,f)$ to be the $r$-partite $r$-uniform multi-hypergraph with sides $\C_1,\dots,\C_r$, where for each $v \in V(G)$ we add the hyperedge $(C_1(v),\dots,C_r(v)) \in \C_1 \times \dots \times \C_r$ (see \cref{fig:dual}).
    Note that $|E(H)| = |V(G)|$, and that
    $d_H(C) = |C|$ for every $C \in V(H)$.
	In what follows, we will denote vertices of $H$ by capital letters, while vertices of $G$ will be denoted by lowercase letters.
	For a vertex $v \in V(G)$, we will denote by $e_v$ the hyperedge of $H$ corresponding to $v$; and vice versa, for a hyperedge $e \in E(H)$, we will denote by $v_e$ the corresponding vertex of $G$.
	We will need the following very simple observation.
	\begin{observation}\label{obs:disj:hyperedges}
	For $u,v \in V(G)$, if $e_u \cap e_v = \es$ then $\{u,v\} \notin E(G)$.
	\end{observation}
	\begin{proof}
	    Suppose that $\{u,v\} \in E(G)$ and let $j \in [r]$ be the colour of $\{u,v\}$. Then $C_j(u) = C_j(v) \in e_u \cap e_v$, so $e_u \cap e_v \neq \es$.
	\end{proof}
	It turns out that the hypergraph construction $H(G,f)$ is precisely what is needed to prove \cref{thm:trees}.
	It is worth noting that this construction has already been used in prior works, see e.g.\ the
	surveys~\cite{Fur88}*{Section 7} and~\cite{Gya11}.

\begin{figure*}[t!]
  \captionsetup{width=0.879\textwidth,font=small}
  \centering
  \includegraphics{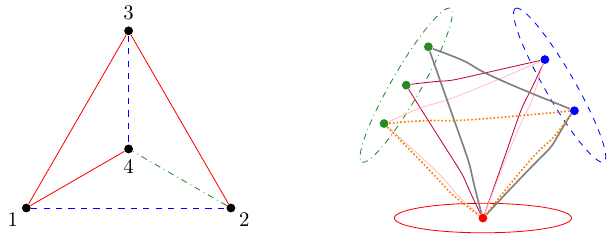}
  \caption{A $3$-coloured $4$-clique and its ``dual'' $3$-partite $3$-uniform multi-hypergraph.}
  \label{fig:dual}
\end{figure*}

A \defn{walk} in a hypergraph $H$ is a sequence of (not necessarily distinct) vertices $v_1,\dots,v_k$ such that for every $1 \leq i \leq k-1$, there is a hyperedge of $H$ containing $v_i,v_{i+1}$.
We say that $H$ is \defn{connected} if there is a walk between any given pair of vertices, or $|V(H)| = 1$.

\begin{lemma}\label{lem:connectivity}
If $G$ is connected then so is $H = H(G,f)$.
\end{lemma}

\begin{proof}
  Let $X,Y \in V(H)$.
  Fix arbitrary $x \in X$, $y \in Y$.
  Since $G$ is connected, it contains a path
  $x = z_0,z_1,\dots,z_k = y$ between $x$ and $y$.
  For each $0 \leq j \leq k-1$, let $i_j$ be the colour of the edge $\{z_j,z_{j+1}\} \in E(G)$, and let $Z_j$ be the colour-$i_j$ component containing this edge.
  Observe that for every $0 \leq j \leq k-2$,
  either $\{z_j,z_{j+1}\}$ and $\{z_{j+1},z_{j+2}\}$ have the same colour,
  in which case $Z_j = Z_{j+1}$,
  or $Z_j$ and $Z_{j+1}$ are contained together in a hyperedge of $H$, namely the hyperedge $e_{z_{j+1}}$ corresponding to the vertex $z_{j+1} \in Z_j \cap Z_{j+1}$.
  Similarly, the hyperedge $e_x$ contains both $X$ and $Z_0$, and the hyperedge $e_y$ contains both $Y$ and $Z_k$ (it is possible that $X = Z_0$ or $Y = Z_k$). It is now easy to see that $X,Z_0,\dots,Z_k,Y$ is a walk in $H$ between $X$ and $Y$, as required.
\end{proof}

We now introduce some additional definitions related to the hypergraph $H = H(G,f)$.
Throughout this section, we assume that $G$ is connected, which in turn implies that $H$ is connected as well (by \cref{lem:connectivity}).
We say that $H$ is \defn{trivial} if $H$ has exactly one vertex in each part.
We assume that $H$ is not trivial (otherwise, $H$ has a monochromatic
spanning tree (in fact, a spanning tree in every colour), and this case will
be trivial when proving \cref{thm:trees}).
A \defn{leaf} is a hyperedge which contains at most one vertex of degree at least $2$. (Observe that $e_v$ is a leaf if and only if all edges of $G$ touching $v$ have the same colour.)
Note that since $H$ is connected and not trivial, every hyperedge of $H$ must contain at least one vertex of degree at least $2$, so a leaf in $H$ contains exactly one such vertex.
Let $H_0 = H_0(G,f)$ be the subhypergraph of $H$ obtained by deleting, for each leaf $e$ of $H$, all $(r-1)$ vertices of $e$ which have degree $1$ (in particular, we delete the hyperedge $e$). Note that deleting a leaf from a connected hypergraph leaves it connected, and hence $H_0$ is connected.
    For $X \in V(H_0)$, denote by $L(X)$ the set of leaves of $H$ in which the unique vertex of degree at least $2$ is $X$.
    \begin{lemma}\label{lem:pruning}
      Suppose that $G$ is $k$-connected, for $k\ge 1$.
      Let $X \in V(H_0)$ be such that $L(X) \ne \es$.
      Then either $d_{H_0}(X) \ge k$, or every hyperedge of $H$ contains $X$.
    \end{lemma}
    \begin{proof}
      Suppose, for the sake of contradiction, that $0\le \ell := d_{H_0}(X) \le k-1$ and that there is a hyperedge of $H$ which does not contain $X$.
      Let $e_1, \ldots, e_\ell \in E(H_0)$ be the only hyperedges of $H_0$ containing $X$.
      For $1 \leq i \leq \ell$, let $v_i \in V(G)$ be such that $e_i = e_{v_i}$, and set $S=\{v_1,\ldots,v_\ell\}$.

      Let $U$ be the set of $u \in V(G)$ such that $e_u \in L(X)$, and note that $U \neq \es$ as $L(X) \neq \es$ by assumption.
      Set $W := V(G) \setminus (U \cup S)$.
      Observe that $W$ is precisely the set of $w \in V(G)$ such that $X \notin e_w$, since every hyperedge which contains $X$ is either in $E(H_0)$ (and hence equals $e_{v_i}$ for some $1\le i\le\ell$) or belongs to $L(X)$.
      It follows that $W \neq \es$, since by assumption there is a hyperedge of $H$ which does not contain $X$.
      We now show that in $G$ there are no edges between $U$ and $W$, which would mean that $S$ is a vertex-cut of $G$, in contradiction to the assumption that $G$ is $k$-connected.
      Let $u \in U, w \in W$. Observe that $e_u \cap e_w = \es$. Indeed, $X
      \notin e_u \cap e_w$, because $X \notin e_w$ as $w \in W$. Also, all
      vertices in $e_u \setminus \{X\}$ have degree one in $H$ (as $e_u \in
      L(X)$), hence they cannot belong to $e_w$. So by \cref{obs:disj:hyperedges}, $\{u,w\} \notin E(G)$, as required.
    \end{proof}
    \noindent Using \cref{lem:pruning}, we now show that if $G$ is $2$-connected then $H_0$ cannot have leaves.
    \begin{lemma}\label{lem:H0:noleaves}
        If $G$ is $2$-connected then $H_0$ has no leaves.
    \end{lemma}
    \begin{proof}
        Suppose by contradiction that $H_0$ has a leaf $e$.
        Since $e$ is not a leaf of $H$ (by definition of $H_0$), there are distinct $X,Y \in e$ with $d_H(X),d_H(Y) \geq 2$.
        Since $e$ is a leaf in $H_0$, at least one of $X,Y$ must have degree $1$ in $H_0$; say $d_{H_0}(X) = 1$ without loss of generality.
        Then $L(X) \neq \es$. By \cref{lem:pruning}, every hyperedge of $H$ contains $X$.
        However, this is impossible.
        Indeed, fix any hyperedge $e' \in E(H) \setminus \{e\}$ with $Y \in e'$; such $e'$ exists because $d_H(Y) \geq 2$.
        Now, if $e'\in E(H_0)$ then $X\notin e'$ since $d_{H_0}(X)=1$.
        And if $e'\notin E(H_0)$ then $e'\in L(Y)$
        and again $X\notin e'$.
        We arrived at a contradiction, proving the lemma.
    \end{proof}
    \noindent
    We are now ready to prove \cref{thm:trees}.

	\begin{proof}[Proof of \cref{thm:trees}]
		Let $G$ be a $2$-connected graph on $n$ vertices.
		It will be convenient to prove the theorem in the following (perhaps slightly convoluted) form: $s_r(G) = O(rd + r^2)$, where $d$ is defined as:
		\[
		d :=
		\begin{cases}
		\Disc_r(G,\mathcal{T}_n) & G \text{ is $3$-connected,} \\
		\max\big\{ \Disc_r(G,\mathcal{T}_n), \sqrt{rn} \big\} & \text{otherwise.}
		\end{cases}
        \]
        Since $\Disc_r(G,\cT_n) \leq d$, there exists an $r$-colouring of the edges of $G$ in which there is no spanning tree with more than $\frac{n-1+d}{r}$ edges of the same colour.
        Fixing one such colouring $f$,
        we claim that
		$|\C_i| \geq \frac{(r-1)n + 1 - d}{r}$ for every $1 \leq i \leq r$ (recall that $\C_i$ is the set of colour-$i$ components). Indeed, by taking a spanning tree of each colour-$i$ component, and connecting these spanning trees using edges of other colours (this is possible because $G$ is connected), we obtain a spanning tree of $G$ with $n - |\C_i|$ edges of colour~$i$. Thus, our assumption implies that
		$n - |\C_i| \leq \frac{n-1+d}{r}$, and hence
		$|\C_i| \geq \frac{(r-1)n + 1- d}{r}$, as claimed.

		Let $H = H(G,f)$ and $H_0 = H_0(G,f)$ be the $r$-partite $r$-uniform multi-hypergraphs defined above.
    Recall that $|E(H)| = |V(G)| = n$.
    By \cref{lem:connectivity}, $H$ is connected (and hence so is $H_0$).
    Trivially, $|V(H)| = |\C_1| + \dots + |\C_r|$.
    As $|\C_i| \geq \frac{(r-1)n - d}{r}$ for every $1 \leq i \leq r$, we have
    \begin{equation}\label{eq:H:vx:edge:count}
        |V(H)| \geq (r-1)n - d.
    \end{equation}
    Observe that since $H_0$ is obtained from $H$ by deleting leaves, we have
    $
    |V(H_0)| = \linebreak |V(H)| - (r-1) \cdot (|E(H)| - |E(H_0)|) = |V(H)| - (r-1)n + (r-1) \cdot |E(H_0)|.
    $
    Now, using~\eqref{eq:H:vx:edge:count}, we get:
	\begin{equation}\label{eq:H':vx:edge:count}
	|V(H_0)| \geq (r-1) \cdot |E(H_0)| - d.
	\end{equation}

    Recalling the statement of \cref{lem:pruning},
    we now observe that
    the second option in the conclusion of the lemma is impossible, as it would imply that one of the parts $\C_1,\dots,\C_r$ contains just one vertex (namely, $X$), in contradiction to the fact that $|\C_i| \geq \frac{(r-1)n + 1 - d}{r} > 1$ for every $1 \leq i \leq r$. Hence, we have the following:
    \begin{claim}\label{cor:3-connectivity_min-degree}
    If $G$ is $3$-connected, then for every $X \in V(H_0)$ with $L(X) \neq \es$, it holds that $d_{H_0}(X) \geq 3$.
    \end{claim}

    Next, we show that by omitting $O(d)$ hyperedges, one can obtain a
    spanning subhypergraph $H_1$ of $H_0$ in which all vertex degrees are not
    larger than $2$, and every hyperedge contains at least $r-2$ vertices of
    degree $1$ (and hence at most $2$ vertices of degree $2$). It is easy to
    see that a hypergraph with these properties is a disjoint union of loose
    paths and cycles\footnote{Recall that an $r$-uniform loose path (resp.\
    cycle) is an $r$-uniform hypergraph obtained from a ($2$-uniform) path
  (resp.\ cycle) by adding $r-2$ ``new" vertices to each of its edges, with distinct edges receiving disjoint sets of new vertices.}.
    We will in fact also guarantee that every vertex $X \in V(H_0)$ with $d_{H_0}(X) \geq 3$ is isolated in $H_1$.
	\begin{claim}\label{claim:cleaning_the_2-core}
		There exists a spanning subhypergraph $H_1$ of $H_0$ with $|E(H_1)| \geq |E(H_0)| - 8d$, having the following properties:
		\begin{enumerate}
			\item The maximum degree in $H_1$ is at most $2$. Furthermore, for every $X \in V(H_0)$, if ${d_{H_0}(X) \ge 3}$ then $d_{H_1}(X) = 0$.
			\item Every hyperedge of $H_1$ contains at least $r-2$ vertices of degree $1$ (in $H_1$).
		\end{enumerate}
	\end{claim}
	\begin{proof}
		Define $A := \{X \in V(H_0) : d_{H_0}(X) = 1\}$ and
		$B := \{X \in V(H_0) : d_{H_0}(X) \geq 2\}$. We have $V(H_0) = A \cup B$. By \cref{lem:H0:noleaves}, every hyperedge of $H_0$ contains at least $2$ vertices from $B$, and hence at most $r-2$ vertices from $A$.
		For every $2 \leq i \leq r$, let $t_i$ be the number of hyperedges of $H_0$ which contain exactly $i$ vertices from $B$ (and hence exactly $r-i$ vertices from $A$).
		Then $t_2 + \dots + t_r = |E(H_0)|$. By the definition of $A$, we have
		\begin{equation}\label{eq:sum_of_degrees_over_A}
		|A| =
		\sum_{i = 2}^{r}{(r-i) \cdot t_i} \leq
		(r-2) \cdot |E(H_0)|.
		\end{equation}
		On the other hand,
		\[
		2|V(H_0)| - |A| = |A| + 2|B| \leq \sum_{X \in V(H_0)}{d_{H_0}(X)} =
		r \cdot |E(H_0)|.
		\]
		So we get that $|A| \geq 2|V(H_0)| - r \cdot |E(H_0)| \geq (r-2) \cdot |E(H_0)| - 2d$, where the last inequality uses~\eqref{eq:H':vx:edge:count}.
        Next, observe that
		\[
		\sum_{i = 3}^{r}{t_i} \leq
		\sum_{i = 3}^{r}{(i-2) \cdot t_i} =
		(r-2) \cdot \sum_{i = 2}^{r}{t_i} - \sum_{i = 2}^{r}{(r-i) \cdot t_i} =
		(r-2) \cdot |E(H_0)| - |A| \leq 2d,
		\]
		where the last equality uses the equality from~\eqref{eq:sum_of_degrees_over_A}.
		So we see that in $H_0$ there are at most $2d$ hyperedges which contain more than $2$ vertices of degree at least $2$.
		Let $E_1$ be the set of such hyperedges, and note that
		$|E_1| \leq 2d$.

		Next, let us handle high-degree vertices. For each $i \geq 1$, let $m_i$ be the number of vertices $X \in V(H_0)$ satisfying $d_{H_0}(X) = i$. Then
		\begin{equation}\label{eq:H_0_degrees_edge-count}
		    \sum_{i \geq 1}{m_i \cdot i} = r \cdot |E(H_0)|
		\end{equation}
		and
		\begin{equation}\label{eq:H_0_degrees_vertex-count}
		    \sum_{i \geq 1}{m_i} = |V(H_0)| \geq (r-1) \cdot |E(H_0)| - d,
		\end{equation}
		where the inequality
		is~\eqref{eq:H':vx:edge:count}.
		Subtracting \eqref{eq:H_0_degrees_vertex-count} from \eqref{eq:H_0_degrees_edge-count}, we obtain
		\begin{equation}\label{eq:H_0_degrees_excess}
		    \sum_{i \geq 2}{m_i \cdot (i-1)} \leq
		    |E(H_0)| + d.
		\end{equation}
		Next, note that $m_1 = |A|$ and hence
		$m_1 \leq (r-2) \cdot |E(H_0)|$ by~\eqref{eq:sum_of_degrees_over_A}.
		Combining this with \eqref{eq:H_0_degrees_vertex-count}, we see that
		$
		\sum_{i \geq 2}{m_i} \geq |E(H_0)| - d.
		$
		Now, subtracting this inequality from \eqref{eq:H_0_degrees_excess}, we obtain
	    \begin{equation}\label{eq:H_0_high-degrees}
	        \sum_{i \geq 2}{m_i \cdot (i-2)} \leq
		    2d.
	    \end{equation}
	    From \eqref{eq:H_0_high-degrees} it follows,  in particular, that
	    $\sum_{i \geq 3}{m_i} \leq 2d$.
        Multiplying this inequality by two and adding the result to \eqref{eq:H_0_high-degrees}, we get that
	    \begin{equation*}
	        \sum_{i \geq 3}{m_i \cdot i} =
	        \sum_{i \geq 3}{m_i \cdot (i-2)} +
	        2 \cdot \sum_{i \geq 3}{m_i} \leq
		    6d.
	    \end{equation*}
		Now observe that $\sum_{i \geq 3}{m_i \cdot i}$ is an upper bound on the number of hyperedges of $H_0$ which contain a vertex of degree at least $3$. Let $E_2$ be the set of such hyperedges (so $|E_2| \leq 6d$). By deleting from $H_0$ the hyperedges in $E_1 \cup E_2$, we obtain a hypergraph satisfying Items 1--2 in the claim. Furthermore, the number of deleted hyperedges is at most $8d$, as required.
	\end{proof}

    For $1 \leq i \leq r$, define
    $L_i := \bigcup_{X \in \C_i \cap V(H_0)}{L(X)}$.
    In other words,
    $L_i$ is the set of leaves of $H$ whose (unique) vertex of degree at least $2$ belongs to $\C_i$.
    Note that
    \begin{equation}\label{eq:leaf_count}
    n = |E(H)| = |E(H_0)| + |L_1| + \dots + |L_r|.
    \end{equation}

    \begin{claim}\label{claim:balanced_number_of_leaves}
    	$|L_i| \leq \frac{n+d}{r}$ for every $1 \leq i \leq r$.
    \end{claim}
  \begin{proof}
    Say that a vertex-set in $G$ is $i$-monochromatic if every edge of $G$ which is incident to a vertex from that set is coloured $i$.
    Observe that if $U\subsetneq V$ is $i$-monochromatic then every spanning tree of $G$ has at least $|U|$ edges coloured $i$.
    Indeed, fix any spanning tree of $G$ and fix a vertex in $V \setminus U$ to be the root of the tree. Orient all edges towards that root.
    Every vertex in $U$ has a unique outgoing edge in the tree, and this edge has colour~$i$.
    Now, let $U$ be the vertices in $G$ corresponding to hyperedges in $L_i$.
    Then $U$ is $i$-monochromatic, hence $|L_i| \le \frac{n-1+d}{r}$.
  \end{proof}

    To complete the proof, we need to find a balanced $r$-separator of $G$ of size $O(rd + r^2)$. In the following claim, we essentially achieve this task by finding a partition of the edges of $H$ which (roughly) corresponds to such a separator. We then explain how to conclude using this claim.
     \begin{claim}\label{claim:main}
    	There exists a partition
    	    $E(H) = E_1 \cup \dots \cup E_r \cup F$ such that:
    	    \begin{enumerate}
    	    	\item $|F| = O(d + r)$, and $|E_i| \leq \frac{n+d}{r} + 2$ for every $1 \leq i \leq r$.
    	    	\item $e_i \cap e_j = \es$ for every $1 \leq i < j \leq r$ and $e_i \in E_i, e_j \in E_j$.
    	    \end{enumerate}
    \end{claim}
	\begin{proof}
		It will be convenient to construct the sets $E_1,\dots,E_r,F$ gradually, i.e.\ by placing various elements in one of these sets at certain stages in the proof. The final sets $E_1,\dots,E_r,F$ form the required partition.

		Let $H_1$ be a spanning subhypergraph of $H_0$ satisfying the assertion of Items 1--2 in \cref{claim:cleaning_the_2-core}. Put all hyperedges in $E(H_0) \setminus E(H_1)$ into $F$. \Cref{claim:cleaning_the_2-core} guarantees that there are at most $8d$ such hyperedges.

		Let $M$ be the set of vertices $X \in V(H_0)$ with $|L(X)| > d/r$. Since $\sum_{X \in V(H_0)}{|L(X)|} = \sum_{i = 1}^{r}{|L_i|} \leq |E(H)| = n$, we have
		$|M| < nr/d$. In particular, if $G$ is not $3$-connected, then our choice of $d$ for that case guarantees that $|M| < d$.
		For each $1 \leq i \leq r$ and $X \in \C_i \cap M$, place into $E_i$ all elements of $L(X)$. Then at the moment we have $E_i \subseteq L_i$, and hence $|E_i| \leq |L_i| \leq \frac{n+d}{r}$ by \cref{claim:balanced_number_of_leaves}. Moreover, the current $E_1,\dots,E_r$ satisfy the assertion of Item~2 because for every $1 \leq i < j \leq r$, no two hyperedges $e_i \in L_i, e_j \in L_j$ intersect.

		Let $H_2$ be the subhypergraph of $H_1$ obtained by deleting from it all vertices belonging to $M$. Put in $F$ all hyperedges of $H_1$ which touch vertices of $M$. Since the maximum degree of $H_1$ is at most $2$ (by Item~1 in \cref{claim:cleaning_the_2-core}), the number of such hyperedges is at most $2|M|$, which is less than $2d$ in the case that $G$ is not $3$-connected.
		If, on the other hand, $G$ is $3$-connected, then there are no hyperedges of $H_1$ whatsoever which touch vertices of $M$. Indeed, this follows from \cref{cor:3-connectivity_min-degree}, which implies that if $X \in M$ then $d_{H_0}(X) \geq 3$, and Item~1 of \cref{claim:cleaning_the_2-core}, which guarantees that $d_{H_1}(X) = 0$ for each such $X$.
		We conclude that in any case, the number of hyperedges added to $F$ at this step is less than $2d$, and hence
		$|F| = O(d)$ at this moment.

		Since $H_2$ is a subhypergraph of $H_1$, it also satisfies the assertion of Items 1--2 in \cref{claim:cleaning_the_2-core}.
		As mentioned above, this means that $H_2$ is a disjoint union of loose paths and cycles (some components may be isolated vertices, and cycles may have length $2$).
		Let $P_1,\dots,P_{\ell}$ be the connected components of
		$H_2$ (each being a loose path or cycle). We now go over $P_1,\dots,P_{\ell}$ in some order, and, when processing $P_k$, do as follows.
		Let $X_1,\dots,X_t$ be the vertices of $P_k$, ordered so that each hyperedge of $P_k$ is a (possibly cyclic) interval in this order (such an ordering exists since $P_k$ is a loose path or cycle).
		Fix $1 \leq i \leq r$ such that $|E_i| \leq \frac{n}{r}$ at this moment (such an $1 \leq i \leq r$ evidently has to exist, as $E_1,\dots,E_r \subseteq E(H)$ are disjoint and $|E(H)| = n$). Let $j$ be the largest integer
		$1 \leq j \leq t$ with the property that adding to $E_i$ all hyperedges in $E' := E(\{X_1,\dots,X_j\}) \cup L(X_1) \cup \dots \cup L(X_j)$ does not increase the size of $E_i$ beyond $\frac{n+d}{r} + 2$. Here $E(\{X_1,\dots,X_j\})$ denotes the set of edges of $P_k$ contained in $\{X_1,\dots,X_j\}$. Observe that $j$ is well-defined, because $|L(X_1)| \leq d/r$ (as $X_1 \notin M$), and because $1 \leq i \leq r$ was chosen so that $E_i$ is not larger than $\frac{n}{r}$ before this step.
		If $j = t$, in which case we added to $E_i$ all edges in $E(P_k) \cup \bigcup_{X\in P_k}{L(X)}$, then we simply continue to the next connected component.
		Suppose now that $j < t$.
		Then it must be the case that after placing $E'$ into $E_i$, the size of $E_i$ exceeds $\frac{n}{r}$, because otherwise we could also add to
		$E_i$ the set $L(X_{j+1})$ and all (at most $2$) hyperedges in $E(\{X_1,\dots,X_{j+1}\})$ containing $X_{j+1}$, thus increasing $|E_i|$ by at most $ |L(X_{j+1})| + 2 \leq d/r + 2$, in contradiction to the maximality of $j$.
		We will say that $E_i$ is \defn{saturated} whenever $|E_i| > \frac{n}{r}$.
		Place in $F$ any hyperedges $e \in E(P_k)$ satisfying $e \cap \{X_1,\dots,X_j\} \neq \es$ and $e \cap \{X_{j+1},\dots,X_t\} \neq \es$, of which there are at most $2$, and put $P_k[\{X_{j+1},\dots,X_t\}]$ into the list of connected components to be processed. The fact that we remove such hyperedges $e$ guarantees that the assertion of Item~2 will be satisfied. Note that if $E_i$ becomes saturated, then no more hyperedges will be added to it at any later stage. Since each $E_i$ ($1 \leq i \leq r$) can become saturated only once, the overall number of edges added to $F$ in this process is at most $2r$. Hence, at the end of the process we have $|F| \leq O(d) + 2r = O(d + r)$. This completes the proof of the claim.
		\end{proof}
		We now complete the proof using \cref{claim:main}.
Let $E(H) = E_1 \cup \dots \cup E_r \cup F$ be a partition satisfying Items 1--2 in that claim. We have
$\min\{|E_1|,\dots,|E_r|\} \geq
n - (r-1) \cdot (\frac{n+d}{r} + 2) - |F| =
\frac{n}{r} - O(d+r)$.
Hence, by moving at most $O(rd + r^2)$ elements from $E_1,\dots,E_r$ to $F$, we may assume that $|E_1| = \dots = |E_r|$.
Now, put $V_i := \{v_e : e \in E_i\}$ (for $1 \leq i \leq r$) and $S := \{v_e : e \in F\}$. Then $|S| = |F| = O(rd + r^2)$ and $|V_1| = \dots = |V_r|$.
For each $1 \leq i < j \leq r$, combining \cref{obs:disj:hyperedges} with
Item 1 in \cref{claim:main} yields that $G$ has no edges between $V_i$ and $V_j$.
It follows that $F$ is a balanced $r$-separator of $G$ of size $O(rd + r^2)$, as required.
This completes the proof.
\end{proof}

\subsection{Tightness: Proof of \cref{prop:tightness}}
The goal of this section is to show that the lower bound on the balanced $r$-separation number of the graph that appears in \cref{thm:trees} is essentially tight.
This is achieved by proving \cref{prop:tightness}.
To this end, we shall construct an $n$-vertex graph $G$ with $s_r(G)=\Omega(\sqrt{n})$ and $\Disc_r(G,\Tn)\le 1$.
The graph will be a \defn{clique cycle}, namely, a cycle (of length $\Theta(\sqrt{n})$) with a disjoint clique attached to each of its edges (see~\cref{fig:cliquecycles}).
Such graphs are obviously $2$-connected.
We will have to choose the sizes of the hanging cliques carefully;
indeed, if the cliques are of equal sizes, say, then one could easily construct a balanced $r$-separator of size $O(r)$.
We will also have to ensure that the clique sizes allow a balanced colouring that guarantees small discrepancy.

For the first task (i.e.,\ forcing a large balanced $r$-separator) we shall use the next technical lemma.
For integers $0\le x<k$ and $0\le i<k$ let $a^{k,x}_i=x$ if $i<x$ and $a^{k,x}_i=k+x$ otherwise.
One can easily verify that for every $0\le x<k$ we have $\sum_{i=0}^{k-1} a^{k,x}_i=k^2$.
Fix $r\ge 2$.
Let $\vect{x}=(x_1,\ldots,x_r)$ be a vector of strictly increasing nonnegative integers, and let $k>x_r$ (so $k\ge r$).
Write $R:=\max\{r,x_r\}$.
Set $\mu=\mu(\vect{x}):=\frac{1}{r}\sum_{j=1}^r x_j$ and assume that $\vect{x}$ is such that $\mu$ is not an integer.
For $0\le i<rk$
write $i=q_ir+t_i$ for $0\le q_i<k$ and $0\le t_i<r$,
and set $b^{k,\vect{x}}_i = a^{k,x_{t_i}}_{q_i}$.
Note that
\begin{equation}\label{eq:bi}
  \sum_{i=0}^{rk-1} b_i^{k,\vect{x}}
  = \sum_{q=0}^{k-1} \sum_{t=0}^{r-1} a_q^{k,t}
  = \sum_{t=0}^{r-1} \sum_{q=0}^{k-1} a_q^{k,t} = rk^2.
\end{equation}
For $I\subseteq\{0,1,\ldots,rk-1\}$ write $\Sigma I=\sum_{i\in I}b^{k,\vect{x}}_i$ for the ``sum'' of $I$,
and let $D(I)=\left| \Sigma I-k^2 \right|$ denote its ``discrepancy'', namely, the deviation of its sum from the ``mean'' $k^2$.
Let $C(I)$ denote the number of connected components of subgraph of the cycle $C_{rk}$ spanned by the vertex set $I$, that is, the number of disjoint consecutive (cyclic) intervals of $I$ in $\mathbb{Z}/(rk)\mathbb{Z}$.
The following lemma shows that every index set either spans many disjoint intervals or has large discrepancy.
\begin{lemma}\label{lem:DC}
  For every $I\subseteq\{0,1,\ldots,rk-1\}$ we have $D(I)+C(I)\cdot R^2\ge k/r-3R^4$.
\end{lemma}
\begin{proof}
  Let $X$ be the set of indices $i\in\{0,1,\ldots,rk-1\}$ satisfying $b^{k,\vect{x}}_i<k$.
  Note that $|X|=r\mu\le R^2$.
  Set $I'=I\sm X$ and note that $D(I)\ge D(I')-R^3$ and $C(I)\ge C(I')-R^2$.
  Observe that for every interval $Y\subseteq\{0,1,\ldots,rk-1\}$ there exists a set $J(Y)\subseteq \{0,\ldots,r-1\}$ and a nonnegative integer $\alpha(Y)$ such that
  $|\{ i\in Y \mid b^{k,\vect{x}}_i=j \}| = \alpha(Y) + \ind_{j\in J(Y)}$.
  Let $Y_1,\ldots,Y_{C(I')}$ be the connected components of $I'$.
  For every subset $J\subseteq\{0,\ldots,r-1\}$ let $A_J$ be the set of components $Y$ of $I'$ for which $J(Y)=J$.
  Note that for $Y\in A_J$ we have
  \[
    \Sigma Y
      =\left(k+\mu\right)|Y|
       + \sum_{j\in J}\left(j-\mu\right),
  \]
  hence
  \[
    \Sigma I' = \sum_J \sum_{Y\in A_J} \Sigma Y
      = \left(k+\mu\right)|I'|
       + \sum_Y\sum_{j\in J(Y)}\left(j-\mu\right).
  \]
  Note also that for each $Y$ we have
  $\left|\sum_{j\in J(Y)}\left(j-\mu\right)\right|\le rx_r \le R^2$, thus
  \[
    \left(k+\mu\right)|I'| - |C(I')|\cdot R^2
    \le \Sigma I' \le
    \left(k+\mu\right)|I'| + |C(I')|\cdot R^2.
  \]
  Suppose that $|I'|\le k-\ceil{\mu}$.
  Then
  \[
    \Sigma I' \le k^2 - k(\ceil{\mu}-\mu) - \mu\ceil{\mu} + C(I')\cdot R^2,
  \]
  hence $D(I') + C(I')\cdot R^2\ge k/r$.
  Suppose now that $|I'|\ge k-\floor{\mu}$.
  Then
  \[
    \Sigma I' \ge k^2 + k(\mu-\floor{\mu}) -\mu\floor{\mu} - C(I')\cdot R^2,
  \]
  hence $D(I') + C(I')\cdot R^2\ge k/r-R^2$.
  To conclude, we have
  \[
    D(I)+C(I)\cdot R^2
      \ge D(I') - R^3 + C(I')\cdot R^2 - R^4
      \ge k/r - R^2 - R^3 - R^4
      \ge k/r - 3R^4.\qedhere
  \]
\end{proof}

Assume from now that $k\gg R$ (more precisely, $R$ is constant and $k\to\infty$).
We proceed by constructing a clique cycle $G$ on $n=rk^2+rk=\Theta_r(k^2)$ vertices with $s_r(G)= \Omega_r(\sqrt{n})$ and $\Disc_r(G,\Tn)\le 1$.
The construction depends on a vector $\vect{x}=(x_1,\ldots,x_r)$ whose entries are strictly increasing nonnegative integers and for which $\mu=\mu(\vect{x})$ is not an integer.
Start with a vertex set $W=\{w_0,\ldots,w_{rk-1}\}$ of size $rk$ (the cycle), and let $A_0,\ldots,A_{rk-1}$ be vertex sets with the following properties:
{\bf (a)} $A_i\cap W=\{w_i,w_{i+1}\}$ for every $0\le i<rk$ (here and in the rest of this section, indices are taken modulo $rk$);
{\bf (b)} $A_i':=A_i\sm \{w_i,w_{i+1}\}$ are pairwise disjoint for $0\le i<rk$;
{and} {\bf (c)} $|A_i'|=b^{k,\vect{x}}_i$.
The vertex set of $G$ is $V(G)=A_0\cup\ldots\cup A_{rk-1}$ (hence, by~\eqref{eq:bi}, $|V(G)|=rk^2+rk$), and $G$ is obtained by letting each $A_i$ ($0\le i<rk$) be a clique, with no futher edges (see \cref{fig:cliquecycles}).
As mentioned earlier, $G$ is $2$-connected.
The proof of \cref{prop:tightness} would follow from the next two claims.

\begin{figure*}[t!]
  \captionsetup{width=0.879\textwidth,font=small}
  \centering
  \begin{subfigure}[t]{0.45\textwidth}
    \centering
    \includegraphics{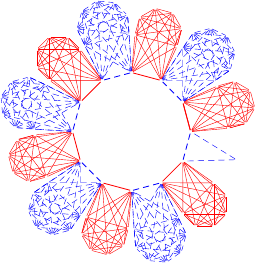}
    \caption{A $2$-coloured clique cycle on $84=2\cdot 6^2+2\cdot 6$ vertices, constructed using $\vect{x}=(0,1)$.}
    \label{fig:cliquecycle}
  \end{subfigure}%
  ~
  \begin{subfigure}[t]{0.05\textwidth}
    ~
  \end{subfigure}%
  ~
  \begin{subfigure}[t]{0.45\textwidth}
    \centering
    \includegraphics{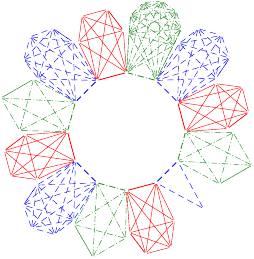}
    \caption{A $3$-coloured clique cycle on $60=3\cdot 4^2+3\cdot 4$ vertices, constructed using $\vect{x}=(0,1,3)$.}
    \label{fig:cliquecycle3}
  \end{subfigure}
  \caption{Clique cycles.}
  \label{fig:cliquecycles}
\end{figure*}

\begin{claim}
  $s_r(G)=\Omega_r(\sqrt{n})$.
\end{claim}

\begin{proof}
  Let $V(G) = V_1 \cup \dots \cup V_r \cup S$ be a balanced $r$-separation
  and let $s:=|S|$.
  For $j\in [r]$ set $I_j$ to be the set of indices $0\le i<rk$ such that $A_i\cap V_j\ne\es$ (and hence $A_i\subseteq V_j\cup S$).
  We may assume without loss of generality that $\Sigma I_j = \sum_{i\in I_j}|A_i'|$ is maximised for $j=1$, and thus $\Sigma I_1 \ge k^2$ and $|I_1|\ge k^2/(k+R)= k (1-o(1))$.
  By \cref{lem:DC} we have that $\Sigma I_1 \ge k^2 + k/r - (C(I_1)+3R^2)R^2$.
  Since $C(I_1)\le 1+|S|$ we have that $\Sigma I_1 \ge k^2 + k/r - (s+4R^2)R^2$.
  It follows that
  \[
    k^2+k-\frac{s}{r} = \frac{n-s}{r} = |V_1| = |V_1\cup S|-|S| \ge \Sigma I_1 + |I_1| - s \ge k^2+\frac{k}{r}-(s+4R^2)R^2+k\cdot(1-o(1))-s,
  \]
  hence
  \[
    k^2+k \ge k^2 + k\cdot\left(\frac{1}{R}+1-o(1)\right)- s\cdot \left(1+R^2-\frac{1}{R}\right) - 4R^4,
  \]
  and therefore
  \[
    s \ge \left(k\cdot\left(\frac{1}{R} -o(1) \right) - 4R^4\right)
    \cdot
    \left(1+R^2-\frac{1}{R}\right)^{-1}
    \ge \frac{k}{R^4} = \Omega_r(\sqrt{n}),
  \]
  where the second inequality holds for large enough $k$.
\end{proof}

\begin{claim}
  $\Disc_r(G,\Tn)\le 1$.
\end{claim}

\begin{proof}
  Define a colouring $f:E(G)\to[r]$ as follows: $f(\{u,v\})=t_i$ if and only if $\{u,v\}$ is an edge of $G[A_i]$.
  We begin by calculating the number of components of each colour.
  Recall that we use $\C_j$ to denote the set of all colour-$j$ components.
  The number of components of colour $j\in[r]$ is the number of cliques coloured $j$ (there are $k$ such cliques), plus the number of vertices which are not in cliques coloured $j$.
  That is,
  \[
    |\cC_j| = k + \sum_{\substack{0\le i<rk\\t_i\ne j}} |A_i'| + (r-2)k
    = (r-1)k^2 + (r-1)k
    = \frac{r-1}{r}\cdot n.
  \]
  Observe that in every spanning tree of $G$, the number of edges of colour~$j$ is at most $n-|\cC_j|=n/r$.
  Thus $\Disc_r(G,\Tn)\le r\cdot(n/r-(n-1)/r)=1$.
\end{proof}

\section{Spanning-Tree Discrepancy in Random Regular Graphs}\label{sec:rrg}
In this section we prove \cref{prop:rrd_tight}.
We begin with the following useful fact, which appears as Corollary~2.15 in~\cite{BKS11}.
\begin{lemma}[\cite{BKS11}]\label{lem:rrg_edge-set_appearance}
  For every $d \geq 3$ there is $C = C(d)$ such that the following holds. Let $G \in \cG_{n,d}$, and let $E_0$ be a set of pairs of vertices of $G$ of size at most $0.49nd$. Then the probability that $E_0 \subseteq E(G)$ is at most
  $\left( \frac{C}{n} \right)^{|E_0|}$.
\end{lemma}
We now use \cref{lem:rrg_edge-set_appearance} to show that \whp{}, all small enough subgraphs of $\cG_{n,d}$ are quite sparse; or, equivalently, that no small edge-set $E_0 \subseteq E(\cG_{n,d})$ spans few vertices.
\begin{lemma}\label{lem:rrg_small_subgraphs_sparse}
  For every $d \geq 3$ and $\eps > 0$ there is $\alpha = \alpha(d,\eps) > 0$ such that \whp{} $G \in \cG_{n,d}$ satisfies the following. For every $E_0 \subseteq E(G)$ of size at most $\alpha n$, the number of vertices spanned by $E_0$ is at least
  $(1 - \eps)|E_0|$.
\end{lemma}
\begin{proof}
  Choose $\alpha = \alpha(d,\eps) > 0$ to satisfy
  $e^2 \cdot C \cdot \alpha^{\eps} \leq \frac{1}{2}$, where $C = C(d)$ is from \cref{lem:rrg_edge-set_appearance}.
  Fixing $1 \leq m \leq \alpha n$, let us estimate the probability that there is a set of less than $(1 - \eps)m$ vertices which span (at least) $m$ edges.
  By the union bound and \cref{lem:rrg_edge-set_appearance}, this probability is at most
  \begin{align*}
  \binom{n}{(1-\eps)m} \cdot \binom{\binom{(1-\eps)m}{2}}{m} \cdot \left(  \frac{C}{n} \right)^m &\leq
  \left(\frac{en}{(1-\eps)m} \right)^{(1-\eps)m} \cdot
  \left(\frac{em}{2}\right)^m \cdot \left(  \frac{C}{n} \right)^m
  \leq
  \left( \frac{e^2 \cdot C \cdot m^{\eps}}{n^{\eps}}\right)^m,
  \end{align*}
  where in the first inequality we used the estimate $\binom{n}{k} \leq \left( \frac{en}{k} \right)^k$ and the (obvious) fact that $\binom{(1-\eps)m}{2} \leq \frac{m^2}{2}$,
  and in the second inequality we used the bound $( \frac{1}{1-\eps} )^{1-\eps} \leq 2$, which holds for every $\eps \in (0,1)$.
  We conclude that the probability that there exists $E_0$ violating the statement of the lemma is at most
  \[
  \sum_{m = 1}^{\alpha n}{\left( \frac{e^2 \cdot C \cdot m^{\eps}}{n^{\eps}}\right)^m}.
  \]
  It is easy to see that for $m$ in the range $1 \leq m \leq \sqrt{n}$, say, the corresponding sum is $o(1)$. As for the range $\sqrt{n} \leq m \leq \alpha n$, we use our choice of $\alpha$ to obtain
  \[
  \sum_{m = \sqrt{n}}^{\alpha n}
  {\left( \frac{e^2 \cdot C \cdot m^{\eps}}{n^{\eps}}\right)^m}
  \leq
  \sum_{m = \sqrt{n}}^{\alpha n}
  {\left( e^2 \cdot C \cdot \alpha^{\eps}\right)^m}
  \leq
  \sum_{m = \sqrt{n}}^{\alpha n}{\left( \frac{1}{2} \right)^m} = o(1),
  \]
  as required. This completes the proof of the lemma.
\end{proof}

Another fact we need regarding random regular graphs concerns the distribution of short cycles. It easily follows from \cref{lem:rrg_edge-set_appearance} that for every fixed $k$, the expected number of $k$-cycles in $\cG_{n,d}$ can be upper bounded by a function of $k$. (In fact, much more precise results are known; see, e.g., Theorem~2.5 in~\cite{Wor99} and the references therein.)
Using Markov's inequality, we obtain the following (relatively weak, though sufficient for our purposes) fact:
\begin{lemma}\label{lem:rrg_short_cycles}
  For every fixed $k$, the number of cycles of length at most $k$ in $\cG_{n,d}$ is $o(n)$ \whp{}.
\end{lemma}

We now combine \cref{lem:rrg_small_subgraphs_sparse,lem:rrg_short_cycles} to show that \whp{}, every small enough edge-set in $\cG_{n,d}$ can be made into (the edge-set of) a forest by omitting only a small fraction of its elements.
\begin{lemma}\label{lem:rrg_almost_forest}
  For every $d \geq 3$ and $\eps > 0$ there is $\beta = \beta(d,\eps) > 0$ such that \whp{} $G \in \cG_{n,d}$ satisfies the following. For every $E^* \subseteq E(G)$ of size at most $\beta n$, there is $F \subseteq E^*$, $|F| \geq (1-\eps)|E^*|$, such that $F$ is the edge-set of a forest.
\end{lemma}
\begin{proof}
  We prove the lemma with $\beta = \beta(d,\eps) = \alpha(d,\frac{\eps}{4})$, where $\alpha$ is from \cref{lem:rrg_small_subgraphs_sparse}.
  Letting $G \in \cG_{n,d}$, we assume that $G$ satisfies the assertions of \cref{lem:rrg_small_subgraphs_sparse} with parameter $\frac{\eps}{2}$ and of
  \cref{lem:rrg_short_cycles} with parameter $k := \frac{3}{\eps}$, and show that under these assumptions, $G$ satisfies the assertion of \cref{lem:rrg_almost_forest}.
  So fix any $E^* \subseteq E(G)$ of size $|E^*| \leq \beta n$.
  Initialise $F$ to $\es$.
  As long as the subgraph spanned by $E^*$ contains a vertex of degree $1$, we delete the edge incident to such a vertex and put it in $F$. Evidently, no edge of this type can be contained in any cycle consisting only of edges from $E^*$. Let $E_0$ be the set of remaining edges at the end of this process, and let $G_0$ be the subgraph spanned by $E_0$. Evidently, $|E_0| \leq |E^*| \leq \beta n \leq \alpha(d,\frac{\eps}{4})n$, so by  \cref{lem:rrg_small_subgraphs_sparse}, $|V(G_0)| \geq (1-\frac{\eps}{4})|E_0|$. Hence, we have
  \begin{equation}\label{eq:rrg-high_degree_vertices}
  \sum_{v \in V(G_0)}{(d_{G_0}(v) - 2)} =
  2|E_0| - 2|V(G_0)| \leq \frac{\eps}{2} |E_0|.
  \end{equation}
  Next, note that by the definition of $E_0$, the minimum degree in $G_0$ is at least $2$, meaning that all summands on the left-hand side in \eqref{eq:rrg-high_degree_vertices} are nonnegative. It now follows from \eqref{eq:rrg-high_degree_vertices} that by deleting at most $\frac{\eps}{2}|E_0|$ edges, we can turn $G_0$ into a graph of maximum degree at most $2$.
  In other words, there exists $E_1 \subseteq E_0$, $|E_1| \geq (1 - \frac{\eps}{2})|E_0|$, such that the maximum degree in the subgraph $G_1$ spanned by $E_1$ is at most $2$. This means that every connected component in $G_1$ is either a cycle or a path. By \cref{lem:rrg_short_cycles}, the number of cycles of length at most $\frac{3}{\eps}$ is $o(n)$. Moreover, the number of cycles of length larger than $\frac{3}{\eps}$ is clearly less than
  $\frac{\eps}{3} |E_1|$. By omitting one edge from each cycle, we obtain an edge-set $E_2 \subseteq E_1$ such that
  $|E_2| \geq (1 - \frac{\eps}{3})|E_1| - o(n) \geq
  (1 - \frac{\eps}{2})|E_1| \geq
  (1 - \frac{\eps}{2})^2|E_0| \geq
  (1 - \eps)|E_0|$, and such that $E_2$ spans a forest. Placing all elements of $E_2$ into $F$, we see that $F$ is a subset of $E^*$ which spans a forest and has size at least
  $|E^* \setminus E_0| + |E_2| \geq
  |E^* \setminus E_0| + (1 - \eps)|E_0| \geq
  (1 - \eps)|E^*|$, as required.
\end{proof}

\begin{proof}[Proof of \cref{prop:rrd_tight}]
  Let $G \in \cG_{n,d}$.
  The upper bound
  $\Disc_r(G,\cT_n) \leq
  \frac{dn}{2} - (n-1) = \linebreak
  \left( \frac{d}{2} - 1 \right) \cdot n + 1$ follows simply by colouring $E(G)$ evenly, namely, such that each colour class has size $\frac{|E(G)|}{r} = \frac{dn}{2r}$. For the lower bound, we show that for every $\eps > 0$, there is $r_0(\eps)$ such that if $r \geq r_0(\eps)$ then \whp{}
  $\Disc_r(G,\cT_n) \geq (d/2 - 1 - \eps)n$. Equivalently, we need to show that \whp{}, in every $r$-colouring of $E(G)$ there is a spanning tree with at least
  $\left( \frac{d}{2r} - \frac{\eps}{r} \right)n$ edges of the same colour. We choose $r_0 = r_0(\eps)$ so that $\frac{d}{2r_0} \leq \beta(d,\frac{2\eps}{d})$, where $\beta$ is from \cref{lem:rrg_almost_forest}. Suppose that $G$ satisfies the assertion of \cref{lem:rrg_almost_forest} (this happens \whp{}). Let $r \geq r_0$, and fix any $r$-colouring of the edges of $G$. By averaging, there is a colour, say $1$, appearing on at least $\frac{|E(G)|}{r} = \frac{dn}{2r}$ of the edges. Fix a set $E^* \subseteq E(G)$ of size $\frac{dn}{2r}$ of edges of colour~$1$. Our choice of $r_0$ guarantees that $|E^*| \leq \beta(d,\frac{2\eps}{d}) \cdot n$. By \cref{lem:rrg_almost_forest}, there is $F \subseteq E^*$  which spans a forest and has size
  $|F| \geq (1 - \frac{2\eps}{d})|E^*|$. By connecting the connected components of this forest (using edges of $G$), we obtain a spanning tree of $G$ with at least
  $|F| \geq (1 - \frac{2\eps}{d})|E^*| =
  (1 - \frac{2\eps}{d}) \cdot \frac{dn}{2r} = \left( \frac{d}{2r} - \frac{\eps}{r} \right)n$ edges of colour~$1$, as required.
\end{proof}

Observe that the only (typical) properties of $\cG_{n,d}$ used in the above proof are that every sufficiently small (linear) edge-set $E_0$ spans at least $(1 - o(1))|E_0|$ vertices (the precise statement is given in \cref{lem:rrg_small_subgraphs_sparse}), and that there are $o(n)$ short cycles (see \cref{lem:rrg_short_cycles}).
Hence, every graph which satisfies the assertions of \cref{lem:rrg_small_subgraphs_sparse,lem:rrg_short_cycles} also satisfies the conclusion of \cref{prop:rrd_tight}. One example of such a graph is the giant component of (a typical) $G(n,p)$ with $p = \frac{d}{n}$ (for $d>1$ and $r$ large enough in terms of $d$).

\section{Proof of Corollaries}
\label{sec:strees:app}

To prove the upper bounds in \cref{cor:hypercube,cor:grid}, it will be convenient to use the following observation.

\begin{observation}\label{obs:ub}
  Let $G=(V,E)$ be a connected graph, and let $V=L_1\cup\ldots\cup L_k$ be a partition of its vertex set such that $|L_i|\le D$ for $i=1,\ldots,k$ and $E(L_i,L_j)=\es$ whenever $1\le i<j-1< k$.
  Then $s_r(G)\le 3rD$.
\end{observation}
\begin{proof}
  We will call the sets $L_1,\dots,L_k$ \defn{layers}.
  We may assume $2D<n/r$, otherwise the statement is trivial.
  A balanced $r$-separator of size $(1+r)D$ can be obtained as follows.
  Let $n=|V|$ and let $\sigma:V\to[n]$ be an ordering of the vertices of $G$ satisfying $\sigma(u)<\sigma(v)$ whenever $u\in L_i$ and $v\in L_j$ for $i<j$.
  For $j=1,\ldots,r-1$ let $L_{i_j}$ be the layer containing the vertex labelled $\floor{jn/r}$.
  Set $V'=V\sm\bigcup_j L_{i_j}$, and observe that $V'$ consists of $r$ parts sized between $n/r-2D$ and $n/r$ each, with no edges between distinct parts.
  We now level all parts by deleting the total of at most $2rD$ vertices, and set $S$ to be the set of vertices outside these parts.
  $S$ is a balanced $r$-separator of $G$ with $|S|\le 3rD$.
\end{proof}

We remark that the ``right'' notion to use here is that of a \defn{bandwidth} of a graph (see, e.g., in~\cite{EncOfAlg}).
The bandwidth of an $n$-vertex graph $G$ is the minimum $D$ such that there exists an ordering $\sigma:V(G)\to[n]$ of the vertices for which $|\sigma(u)-\sigma(v)|\le D$ for every edge $\{u,v\}$, namely,
\[
  \bw(G) = \min_\sigma \max_{\{u,v\}\in E(G)}|\sigma(u)-\sigma(v)|,
\]
where the minimum is over all orderings $\sigma$ of $V(G)$.
Using this terminology, the statement of \cref{obs:ub} can be simplified as follows:
for every graph $G$ with $\bw(G)\le D$ it holds that $s_r(G)\le 3rD$.

\paragraph{The hypercube}
Here we prove \cref{cor:hypercube}.
Identify the vertices of $Q_d$ with the set $\{0,1\}^d$ in the obvious way.
Denote by $L_0,\ldots,L_d$ the layers of the hypercube, namely, $L_i=\{(x_1,\ldots,x_d)\in\{0,1\}^d\mid \sum_{j=1}^d x_j= i\}$. Evidently, $|L_i| = \binom{d}{i}$.
For the lower bound in \cref{cor:hypercube}, we shall need the following lemma.

\begin{lemma}\label{lem:hypercube:iso}
  For every $\alpha>0$ there exists $\beta>0$ such that for every set $U$ of size at least $\alpha\cdot 2^d$ it holds that $|N(U)|\ge \beta\cdot 2^d/\sqrt{d}$.
\end{lemma}

\begin{proof}
  Fix $\alpha>0$ and let $U$ be a vertex set with $|U|\ge \alpha\cdot 2^d$.
  A \defn{Hamming ball} (with centre $\mathbf{0}$ and \defn{radius} $k$) in $Q_d$ is a vertex set of the form $L_0\cup L_1\ldots\cup L_{k-1}\cup L_k'$ for some $0\le k\le d$ and $\es\ne L_k'\subseteq L_k$.
  A classical result by Harper (\cite{Har66}, see also, e.g.,\ Theorem~31 of~\cite{GP}) implies that the vertex boundary of sets of a given size is minimised by Hamming balls.
  Thus, we may assume $U$ is a Hamming ball (of radius $k$).
  Note that (e.g.,~by Chernoff bounds) $k\ge d/2-\alpha'\sqrt{d}$ for some $\alpha'=\alpha'(\alpha)$.
  Now, note that $|N(U)|$ is at least asymptotically half of the size of $L_k$, which is at least $\beta\cdot 2^d/\sqrt{d}$ for some $\beta=\beta(\alpha)$, as required.
\end{proof}

\begin{proof}[Proof of \cref{cor:hypercube}]
  From \cref{lem:hypercube:iso} we see that $s_r(Q_d)\succeq 2^d/\sqrt{d}$ (using the same logic as in the proof of \cref{cor:iso}).
  Thus, as $Q_d$ is $2$-connected (for $d\ge 2$), \cref{thm:trees} implies that $\Disc_r(Q_d,\Tn)\succeq 2^d/\sqrt{d}$.
  The upper bound is obtained by combining \eqref{eq:ub} with \cref{obs:ub} and by noting that $|L_i|\preceq 2^d/\sqrt{d} = n/\sqrt{\log n}$ for every $i=0,\ldots,d$.
\end{proof}

\paragraph{The grid}
Here we prove \cref{cor:grid}.
Identify the vertices of $P_k^d$ with the set $[k]^d$ in the obvious way.
Denote by $L_1,\ldots,L_k$ the $k$ layers of the grid, each spans a copy of $P_k^{d-1}$, namely, $L_i=\{(x_1,\ldots,x_{d-1},i)\mid x_j\in [k],\ j=1,\ldots,d-1\}$.
\begin{proof}[Proof of \cref{cor:grid}]
  From~\cite{BL91v} (see also~\cite{WW77}) we know that $\iota(P_k^d)\succeq k^{d-1}$.
  The lower bound for $d\ge 3$ thus follows from \cref{cor:iso} by noting that $P_k^d$ is $3$-connected.
  For $d=2$, let $P_k^+$ be obtained from $P_k^2$ by adding a cycle through the ``corner'' vertices $(1,1),(1,k),(k,k),(k,1)$.
  Note that any spanning tree of $P_k^+$ contains at most $3$ edges which are not in $P_k^2$.
  As $P_k^+$ is clearly $3$-connected, we have by \cref{cor:iso} that $\Disc_r(P_k^2,\Tn)\ge \Disc_r(P_k^+,\Tn)-3r \succeq k$.
  The upper bound (for $d\ge 2$) is obtained from the combination of \eqref{eq:ub} and \cref{obs:ub} by noting that $|L_i|=k^{d-1}$ for every $i=1,\ldots,k$.
\end{proof}

\section{Spanning-Tree Discrepancy in \texorpdfstring{$\beta$-Graphs}{Beta-Graphs}}
\label{sec:beta}
\paragraph{Proof of \cref{thm:beta}}
Fix any $\eps > 0$.
Set $m = \lceil 2^{r+1}/\eps \rceil$ and
$\beta = \beta(r,\eps) := \frac{\eps}{16rm}$.
Let $G$ be a connected $n$-vertex $\beta$-graph, and fix any $r$-colouring of the edges of $G$. Recall that our goal is to show that there is a spanning tree of $G$ with at least $(\varphi(r) - \eps)n$ edges of the same colour. It will be convenient to assume that $m$ divides $n$. If not, then take a connected induced subgraph $G'$ of $G$ on $n' := \lfloor n/m \rfloor \cdot m > n - m$ vertices. Then $G'$ is a $2\beta$-graph because $2\beta n' \geq \beta n$. Also, a spanning tree of $G'$ with at least $(\varphi(r) - \eps)n'$ edges of the same colour can be extended to a spanning tree of $G$ with at least $(\varphi(r) - \eps)n' \geq (\varphi(r) - 2\eps)n$ edges of the same colour. So it suffices to prove the assertion in the case that $m$ divides $n$.

  For each $1 \leq i \leq r$, let $G_i$ be the graph on $V(G)$ consisting of the edges coloured with colour~$i$.
  Connected components of $G_i$ will be called \defn{colour-$i$ components}, and their number will be denoted by $\ell_i$.
  Suppose first that $\ell_i \leq (1 - \varphi(r) + \eps) \cdot n$ for some $1 \leq i \leq r$. In this case, take a spanning forest of each colour-$i$ component, and connect these components using $\ell_i - 1$ edges to obtain a spanning tree $T$ of $G$ (this is possible since $G$ is connected). The number of edges of $T$ of colour~$i$ is at least
  $n - 1 - (\ell_i - 1) \geq n - (1 - \varphi(r) + \eps) \cdot n = (\varphi(r) - \eps) \cdot n$, as required.
  So we see that in order to complete the proof, it suffices to rule out the possibility of having $\ell_i > (1 - \varphi(r) + \eps) \cdot n$ for all $1 \leq i \leq r$. Suppose then, for the sake of contradiction, that this is the case. For $1 \leq i \leq r$, let $V_i$ be the union of all colour-$i$ components of size at most $2/\eps$. Evidently, the number of colour-$i$ components of size larger than $2/\eps$ is less than $\eps n/2$. Since $|V_i|$ is at least as large as the number of colour-$i$ components of size at most $2/\eps$, we have
  \begin{equation*}\label{eq:beta:Vi}
  |V_i| > \ell_i - \eps n/2 > (1 - \varphi(r) + \eps/2) \cdot n,
  \end{equation*}
  where in the last inequality we used our assumption that $\ell_i > (1 - \varphi(r) + \eps) \cdot n$.
   For each $1 \leq i \leq r$, put $W_i := V(G) \setminus V_i$, noting that
	\begin{equation}\label{eq:size_of_cliques}
	|W_i| = n - |V_i| <
	(\varphi(r) - \eps/2) \cdot n.
	\end{equation}

	Now, consider the Venn diagram of the sets $W_1,\dots,W_r$. Partition each of the (at most) $2^r$ ``regions" of this Venn diagram into sets of size $n/m$, plus a ``residual set" of size less than $n/m$. Then, collect all residual sets and partition their union into sets of size $n/m$. Let $U_1,\dots,U_m$ be the resulting partition of $V(G)$. Note that for each $1 \leq i \leq r$ and for all but at most $2^r$ of the indices $1 \leq s \leq m$, it holds that either $U_s \subseteq W_i$ or $U_s \cap W_i = \es$. Indeed, if $U_s$ is not contained in the union of residual sets, then $U_s$ is contained in one of the regions of the Venn diagram of $W_1,\dots,W_r$, implying that either $U_s \subseteq W_i$ or $U_s \cap W_i = \es$.
	\begin{claim}\label{claim:approximate_covering}
		For every pair $1 \leq s < t \leq m$, there exists $1 \leq i \leq r$ such that $|U_s \cap W_i| \geq \frac{n}{2mr}$ and $|U_t \cap W_i| \geq \frac{n}{2mr}$.
	\end{claim}
	\begin{proof}
		Let $1 \leq s < t \leq m$. Suppose, for the sake of contradiction, that for each $1 \leq i \leq r$ it holds that $|U_s \cap W_i| < \frac{n}{2mr}$ or $|U_t \cap W_i| < \frac{n}{2mr}$. We now define subsets $X \subseteq U_s$ and $Y \subseteq U_t$, as follows. For each $1 \leq i \leq r$, if $|U_s \cap W_i| < \frac{n}{2mr}$ then remove the elements of $U_s \cap W_i$ from $U_s$, and if $|U_t \cap W_i| < \frac{n}{2mr}$ then remove the elements of $U_t \cap W_i$ from $U_t$. Let $X$ be the set of remaining elements in $U_s$ and $Y$ be the set of remaining elements in $U_t$. By definition, $X \cap Y = \es$. Moreover, we have
		$|X| > |U_s| - r \cdot \frac{n}{2mr} = \frac{n}{2m}$ and similarly
		$|Y| > |U_t| - r \cdot \frac{n}{2mr} = \frac{n}{2m}$.
		Let us estimate the number of edges (in $G$) between $X$ and $Y$. To this end, fix $1 \leq i \leq r$ and suppose, without loss of generality, that $X \cap W_i = \es$ (the case $Y \cap W_i = \es$ is symmetrical). The definitions of $W_i$ and $V_i$ imply that for every $x \in X$, the colour-$i$ component containing $x$ has size at most $2/\eps$. This means that for every $x \in X$, there are less than $2/\eps$ edges of colour~$i$ incident to $x$. Hence,
		$e_{G_i}(X,Y) < 2/\eps \cdot |X| \leq 2/\eps \cdot |U_s| = \frac{2n}{\eps m}$.
		Summing over all colours $1 \leq i \leq r$, we conclude that
		\begin{equation}\label{eq:eXY:upper:bd}
		e_G(X,Y) = \sum_{i = 1}^{r}{e_{G_i}(X,Y)} < \frac{2rn}{\eps m}.
		\end{equation}
		On the other hand, since $G$ is a $\beta$-graph, there are at least $|Y| - \beta n$ edges between $X'$ and $Y$ for every subset $X' \subseteq X$ of size $\beta n$. Hence, by considering a partition of $X$ into sets of size $\beta n$, we see that
		\begin{equation}\label{eq:eXY:lower:bd}
		e_G(X,Y) \geq \left\lfloor \frac{|X|}{\beta n} \right\rfloor \cdot (|Y| - \beta n) \geq
		\frac{1}{2m\beta} \cdot \left( \frac{1}{2m} - \beta \right) \cdot n
		\geq
		\frac{n}{8m^2\beta},
		\end{equation}
		where in the second inequality we used the fact that $|X|,|Y| \geq \frac{n}{2m}$, and in the last inequality we used the fact that $\beta \leq \frac{1}{4m}$, which follows from our choice of $\beta$. By combining~\eqref{eq:eXY:upper:bd} and~\eqref{eq:eXY:lower:bd}, we get
		$
		\frac{1}{8m^2\beta} < \frac{2r}{\eps m}$, or, equivalently,
        $\beta > \frac{\eps}{16rm}$.
		But this contradicts our choice of $\beta$.
	\end{proof}
	Let us now complete the proof of the theorem using \cref{claim:approximate_covering}.
	For each $1 \leq i \leq r$, let $A_i$ be the set of all $s \in [m]$ such that $|U_s \cap W_i| \geq \frac{n}{2rm}$. By \cref{claim:approximate_covering}, for every $e = \{s,t\} \in \binom{[m]}{2}$ it holds that $e \subseteq A_i$ for some $1 \leq i \leq r$. In other words, $A_1,\dots,A_r$ is a covering by cliques of the edges of the complete graph on $[m]$. Hence, the definition of $\varphi(r,m)$ implies that
	\begin{equation}\label{eq:beta:Ai:lower:bd}
	\max_{1 \leq i \leq r}{|A_i|} \geq \varphi(r,m) \geq \varphi(r)  \cdot m,
	\end{equation}
  where the last inequality holds (for every $m$) by the results of \cites{HS92,Mil79}.

    On the other hand, fixing any $1 \leq i \leq r$, recall that for all but at most $2^r$ of the indices $1 \leq s \leq m$ it holds that either $U_s \subseteq W_i$ or $U_s \cap W_i = \es$. This implies that
	$|A_i| \leq |W_i| \cdot m/n + 2^r$.
	Combining this with~\eqref{eq:size_of_cliques}, we get that
	\begin{equation}\label{eq:beta:Ai:upper:bd}
	|A_i| < (\varphi(r) - \eps/2) \cdot m + 2^r
	\leq \varphi(r) \cdot m
	\end{equation}
	for every $1 \leq i \leq r$.
	But \eqref{eq:beta:Ai:upper:bd} stands in contradiction with \eqref{eq:beta:Ai:lower:bd}. This completes the proof.
\qed

\vspace{0.3cm}

Note that in the above proof, $\beta(r,\eps)$ is chosen as
$\beta(r,\eps) = \Theta_r(\eps^2)$. As $K_n$ is a $\beta$-graph
with $\beta = 1/n$, we may apply \cref{thm:beta} to $K_n$ with $\eps = \Omega_r(\frac{1}{\sqrt{n}})$, obtaining that in every $r$-colouring of $E(K_n)$ there is a spanning tree with at least $\varphi(r) \cdot n - O_r(\sqrt{n})$ edges of the same colour. As mentioned in the introduction, it would be interesting to obtain an exact result for the complete graph.

\section{Discrepancy of Hamilton Cycles}
\label{sec:HamCycles}
The goal of this section is to prove \cref{thm:ham}.
In the proof, we will use the following simple gadget. Let $H$ be the graph obtained by gluing two $4$-cycles at a vertex. So $H$ has $7$ vertices and $8$ edges.
We say that an edge-coloured copy of $H$ is {\em well-coloured} if there are two distinct colours, say $i$ and $j$, such that the vertex of degree four is incident to two edges of colour $i$ and two edges of colour $j$, and the two edges of each of the colours come from different $4$-cycles (there are no restrictions on the colours of the edges not touching the vertex of degree four). See Figure \ref{fig:gadget}. A crucial property of a well-coloured copy of $H$ is that the set of edges of any given colour spans a path forest.
We start by showing that in every graph with a minimum degree of slightly above $n/2$, one can find a well-coloured copy of $H$ or a long monochromatic \nolinebreak cycle.

\begin{figure}
  \captionsetup{width=0.879\textwidth,font=small}
    \centering
    \includegraphics{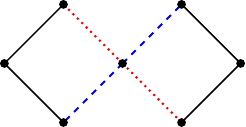}
    \caption{A well-coloured copy of $H$.}
    \label{fig:gadget}
\end{figure}

\begin{lemma}\label{lem:gadget or cycle}
Let $r \geq 2$, let $\eps > 0$, and let $G$ be a graph on $n$ vertices with $\delta(G)\ge (1/2 + \eps)n$, where $n \geq n_0(r,\eps)$. Then for every $r$-colouring of $E(G)$, there is a well-coloured copy of $H$ or a monochromatic cycle of length at least $(1/2 + \eps/2)n$.
\end{lemma}

\begin{proof}
Fix any $r$-colouring of the edges of $G$.
For $1\le i\le r$, define $V_i$ to be the set of all vertices $v\in V$ such that all but at most $r-1$ edges incident to $v$ are coloured in colour $i$. Let also $V_0=V(G) \setminus \bigcup_{i=1}^rV_i$.

If $V_0\ne\es$, then let $v\in V_0$, and let $i$ be a majority colour at $v$. Since $v\not\in V_i$, there are at least $r$ edges of $G$ incident to $v$ and not coloured $i$, and hence there is a colour $j\ne i$ with at least two edges incident to $v$ and coloured $j$; so let $w_1,w_2 \in V(G) \setminus \{v\}$ be such that $\{v,w_1\},\{v,w_2\}$ are coloured with colour $j$. Since $v$ has at least
$d(v)/r \geq (1/2+\eps)n/r \geq 2$ neighbours in colour $i$, there are distinct $u_1,u_2 \in V(G)$ such that $\{v,u_1\},\{v,u_2\}$ are coloured with colour $i$. Now, as $\delta(G) \geq (1/2+\eps)n$, every pair of vertices has at least $2\eps n \geq 5$ common neighbours. Hence, there are distinct $x_1,x_2 \notin \{v,u_1,u_2,w_1,w_2\}$ such that $x_k$ is adjacent to $u_k,w_k$ ($k = 1,2$). Now $v,u_1,u_2,w_1,w_2,x_1,x_2$ form a well-coloured copy of $H$, as required.

So from now on we can assume $V_0=\es$, hence $V(G) = V_1 \cup \dots \cup V_r$. Without loss of generality, assume that $|V_1| \geq n/r$. Fix any $2\le j\le r$.
By definition, all but at most $(r-1)|V_1|$ of the edges between $V_1$ and $V_j$ are coloured in $1$, and all but at most $(r-1)|V_j|$ of these edges are coloured in $j$. It follows that $|E(V_1,V_j)|\le (r-1)(|V_1|+|V_j|)$. Now, summing over all $j = 2,\dots,r$, we get that
\begin{equation}\label{eq:gadget or cycle}
\begin{split}
e(V_1,V(G) \setminus V_1)&\le
(r-1)^2|V_1| + (r-1)(n - |V_1|) = (r-1)(r-2)|V_1| + (r-1)n \\ &\leq 2(r-1)^2|V_1|.
\end{split}
\end{equation}
By averaging, there is $u \in V_1$ which sends at most $2(r-1)^2$ edges to $V(G) \setminus V_1$. Since $d(u) \geq (1/2 + \eps)n$, we have $|V_1| \geq (1/2 + \eps)n - 2(r-1)^2$.

Let $V'_1$ be the set of all $v \in V_1$ which send at most $(r-1)\sqrt{n}$ edges to $V(G) \setminus V_1$. By \eqref{eq:gadget or cycle} we have $|V_1 \setminus V'_1| < 2(r-1)|V_1|/\sqrt{n} \leq 2(r-1)\sqrt{n}$. For each $v \in V'_1$, delete all at most $r-1$ edges incident to $v$ which do not have colour $1$. In the resulting graph, each $v \in V'_1$ still has at least $(1/2+\eps)n - 3(r-1)\sqrt{n} - (r-1) \geq n/2 \geq |V'_1|/2$ neighbours inside $V'_1$ in the colour $1$. By Dirac's theorem, $G[V'_1]$ contains a Hamilton cycle in colour $1$, giving a monochromatic cycle of length $|V'_1| \geq |V_1| - 2(r-1)\sqrt{n} \geq (1/2+\eps)n - 2(r-1)^2 - 2(r-1)\sqrt{n} \geq (1/2+\eps/2)n$, as required.
\end{proof}

Next, we use \cref{lem:gadget or cycle} to show that in a graph with minimum degree significantly larger than $\frac{n}{2}$, one can find a monochromatic path forest which is significantly larger than $\frac{n}{r}$.

\begin{lemma}\label{lem:monochromatic path forest}
Let $r \geq 2$, let $\eps > 0$, and let $G$ be a graph on $n$ vertices with $\delta(G)\ge (1/2 + \eps)n$, where $n \geq n_0(r,\eps)$. Then for every $r$-colouring of $E(G)$, there is a monochromatic path forest of size at least
$\left(1+\frac{\eps}{7}\right)\frac{n}{r}$.
\end{lemma}

\begin{proof}
Let $k=\eps n/7$.
Applying \cref{lem:gadget or cycle} repeatedly, we find a monochromatic cycle/path of length at least $n/2$ in $G$
--- in which case we are done ---
or $k$ vertex-disjoint well-coloured copies $H_1,\ldots, H_k$ of $H$.
Indeed, after finding such a copy, we delete its vertices from the graph and continue. After $j$ steps, the graph has $n - 7j$ vertices and minimum degree at least $(1/2+\eps)n - 7j \geq (1/2+\eps/2) \cdot (n-7j)$, provided $j \leq k$.
We can therefore continue this process for $k$ steps.
Let $U=V-\bigcup_{j=1}^k V(H_j)$. Then $|U|=n-7k$, and $\delta(G[U])\ge (1/2+\eps)n - 7k \geq |U|/2$. So by Dirac's theorem, $U$ spans a Hamilton cycle $C_0$. We can assume $C_0$ is not monochromatic as otherwise we immediately get a monochromatic path of length at least $n/2$. Observe that for every colour $1\le i\le r$, the edges coloured by $i$ in all copies $H_j$ and the cycle $C_0$ form a monochromatic path forest. Indeed, this follows from the definition of a well-coloured copy of $H$.
Let $t_{i,j}$ be the number of edges of colour $i$ in $H_j$, and let $t_{i,0}$ be the number of edges of colour $i$ in $C_0$. Then we have
$$
\sum_{i=1}^r\sum_{j=1}^k t_{i,j}+\sum_{i=1}^rt_{i,0}=8k+n-7k=n+k\,,
$$
implying that one of these monochromatic path forests has a size of at least $(n+k)/r=\left(1+\frac{\eps}{7}\right)\frac{n}{r}$ as required.
\end{proof}

We note that Lemma \ref{lem:monochromatic path forest} is tight in the sense that minimum degree $\frac{n}{2}$ is not sufficient to force a monochromatic path forest of size larger than $\frac{n}{r}$. To see this, take a complete $\frac{n}{2} \times \frac{n}{2}$ bipartite graph with sides $A,B$. Partition $A$ into sets $A_1,\dots,A_r$ of size $\frac{n}{2r}$ each, and colour all edges between $A_i$ and $B$ with colour $i$. Then the largest colour-$i$ matching has size $\frac{n}{2r}$, and hence the largest colour-$i$ path forest has $\frac{n}{r}$ vertices.

The last tool we need to prove Theorem \ref{thm:ham} is the following lemma due to P\'osa~\cite{Pos63}.
\begin{lemma}[\cite{Pos63}]\label{lem:path_extension_to_Hamilton_cycle}
	Let $t \geq 0$ and let $G$ be a graph with $n$ vertices and minimum degree at least $\frac{n}{2} + t$. Let $E \subseteq E(G)$ be an edge-set which forms a path-forest and has size at most $2t$. Then there exists a Hamilton cycle in $G$ which uses all edges in $E$.
\end{lemma}
\begin{proof}[Proof of \cref{thm:ham}]
	Let $G$ be an $n$-vertex graph with $\delta(G) \geq \frac{(r+1)n}{2r} + d$, and fix any $r$-colouring of the edges of $G$. By Lemma \ref{lem:monochromatic path forest} with $\eps = \frac{1}{2r}$, there exists a monochromatic path forest $F$ of size $|F| \geq \left(1+\frac{\eps}{7}\right)\frac{n}{r} = \left( \frac{1}{r} + \frac{1}{14r^2} \right)n$. Take $E \subseteq F$ of size $|E| = \frac{n}{r} + 2d$; this is possible as $d \leq \frac{n}{28r^2}$ by assumption.
	By \cref{lem:path_extension_to_Hamilton_cycle} with $t = \frac{n}{2r} + d$, $G$ contains a Hamilton cycle which uses all edges in $E$, and hence has at least
	$|E| = \frac{n}{r} + 2d$
	edges of the same colour. This completes the proof.
\end{proof}

Observe that we use the high minimum degree of $G$ in the above proof in two places:
first, to obtain a large (i.e., of size significantly larger than $\frac{n}{r}$) monochromatic path-forest, and second, to incorporate this path-forest in a Hamilton cycle.
For the first part, it suffices that the minimum degree is (notably) larger
than $\frac{n}{2}$ (see \cref{lem:monochromatic path forest}).
It is in the second part that we use the assumption that the minimum degree is actually larger than $\frac{(r+1)n}{2r}$.

\begin{acknowledgements}
The authors wish to thank Matan Harel, Gal Kronenberg and Yinon Spinka for valuable discussions.
\end{acknowledgements}

\bibliography{library}

\end{document}